\documentclass{amsart}

\usepackage{graphicx}

\usepackage{amssymb}
\usepackage{amsmath}

\newtheorem{theorem}{Theorem}
\newtheorem{lemma}[theorem]{Lemma}
\newtheorem{corollary}[theorem]{Corollary}
\newtheorem{proposition}[theorem]{Proposition}

\newtheorem{definition}[theorem]{Definition}

\newtheorem{fact}{Fact}

\def\sin{C}
\def\Msin{M}

\def\sout{D}

\def\cT{\mathcal{T}}
\def\cE{\mathcal{E}}
\def\uar{\mathrm{Out}\ \!}
\def\dar{\mathrm{In}\ \!}
\def\Seq{\mathrm{Seq}}

\def\ol{\overline}

\def\wh{\widehat}
\def\oC{G_2'}

\begin{document}

\medskip

\title[Schnyder woods for higher genus triangulated surfaces]{Schnyder woods for higher genus triangulated surfaces, with applications
       to encoding}

\medskip

\author[L. Castelli Aleardi]{Luca Castelli Aleardi}
\address{L. Castelli Aleardi: 
LIX, \'Ecole Polytechnique, 91128 Palaiseau Cedex, France\footnote{Part of the first author's work was done during his visit to the CS Department of Universit\'e Libre de Bruxelles
(Belgium).}}   
\email{amturing@lix.polytechnique.fr}
\author[\'E. Fusy]{\'Eric Fusy}
\address{\'E. Fusy: LIX, \'Ecole Polytechnique, 91128 Palaiseau Cedex, France\footnote{Part of the second author's work was done during his visit to the math department of University of British Columbia
(Vancouver, Canada).}}
\email{fusy@lix.polytechnique.fr}

\author[T. Lewiner]{Thomas Lewiner}
\address{T. Lewiner: Department of Mathematics, PUC-Rio, Brazil}   
\email{lewiner@gmail.com}

\date{\today}

\begin{abstract}
Schnyder woods are a well-known combinatorial structure for plane
triangulations, which
yields a decomposition into 3 spanning trees.
We extend here
definitions and algorithms for Schnyder woods to closed orientable
surfaces of arbitrary genus.
In particular, we describe a method to traverse a triangulation of genus $g$
and compute a so-called $g$-Schnyder wood on the way. As an application,
we give a procedure to encode a triangulation of genus $g$ and $n$ vertices
in $4n+O(g \log(n))$ bits. This matches the worst-case encoding rate of Edgebreaker
in positive genus. All the algorithms presented here have execution time $O((n+g)g)$,
hence are linear when the genus is fixed.

\emph{This is the extended and revised journal version of a
    conference paper with the title ``Schnyder woods generalized to higher
    genus triangulated surfaces'', which appeared in the Proceedings
    of the ACM Symposium on Computational Geometry 2008 (pages 311-319).}
\end{abstract}

\maketitle


\section{Introduction}\label{sec:intro}
Schnyder woods are a nice and deep combinatorial structure to
finely capture the notion of planarity of a graph. They are named
after W. Schnyder, who introduced these structures under the name
of realizers and derived as main applications a new planarity
criterion in terms of poset dimensions~\cite{Schnyder89}, as well
as a very elegant and simple straight-line drawing
algorithm~\cite{Schnyder90}.
There are several equivalent formulations of Schnyder woods,
either in terms of \emph{angle labeling} (Schnyder labeling) or
\emph{edge  coloring and orientation} or in terms of orientations
with prescribed out-degrees.
The most classical formulation is for the family of maximal plane
graphs, i.e., plane triangulations, yielding the following
striking property: the internal edges of a triangulation can be
partitioned into three trees that span all inner vertices and are
 rooted respectively at each
of the three vertices incident to the outer face. Schnyder woods,
and more generally $\alpha$-orientations, received a great deal of
attention~\cite{Schnyder90,Felsner01,Kant96,Fusy2007_thesis}. From
the combinatorial point of view, the set of Schnyder woods of a
fixed triangulation has an interesting lattice
structure~\cite{Brehm_thesis,Bonichon2002_thesis,Felsner04,DeFraysseix01,Ossona1994_thesis},
and the nice characterization in terms of spanning trees motivated
a large number of applications in several domains such as graph
drawing~\cite{Schnyder90,Kant96}, graph coding and
random sampling~\cite{Chu98,He99,Bon03,Pou03,Fus05,Bon06,CastelliDS06,BarbayISAAC07}.
Previous work focused mainly on the application and extension of the combinatorial properties
of Schnyder woods to 3-connected plane
graphs~\cite{Felsner01,Kant96}.
In this article,  we focus on triangulations, but, which is new, we consider
triangulations in arbitrary genus.

\newpage

\subsection{Related Work}\label{sec:related_work}

\subsubsection{Vertex spanning tree decompositions}
In the area of tree decompositions of graphs there exist some
works dealing with the higher genus case.
We mention one recent attempt to generalize Schnyder woods to the
case of toroidal graphs~\cite{Bon05} (genus $1$ surfaces), based
on a special planarization procedure.
In the genus $1$ case it is actually possible to choose two
adjacent non-contractible cycles, defining a so-called
\emph{tambourine}, whose removal makes the graph planar; the graph
obtained can thus be endowed with a Schnyder wood.
In the triangular case this approach yields a process for
computing a partition of the edges into three edge-disjoint
spanning trees plus at most $3$ edges.
Unfortunately, as pointed out by the authors, the local conditions
of Schnyder woods are possibly not satisfied for a large number of
vertices, because the size of the tambourine might be arbitrary
large. Moreover, it is not clear how to generalize the method to
genus $g\geq 2$.


\subsubsection{Planarizing graphs on surfaces}
A possible solution to deal with Schnyder woods (designed originally for
plane triangulations) in higher genus would consist in performing a
planarization of the surface.
Actually, given a triangulation $\mathcal{T}$ with $n$ vertices
on a surface $\mathcal{S}$
of genus $g$,
one can compute a cut-graph or a collection of $2g$ non-trivial
cycles, whose removal makes $\mathcal{S}$ a topological disk
(possibly with boundaries).
There is a number of recent
contributions~\cite{CabelloM07,VerdiereL02,EricksonH04,Kutz06,LazarusPVV01,VegterY90}
for the
efficient computation of cut-graphs, optimal (canonical) polygonal
schemas and shortest non-trivial cycles.
For example some work makes it possible to compute polygonal
schemas in time $O(gn)$ for a triangulated orientable
manifold~\cite{LazarusPVV01,VegterY90}.
Nevertheless we point out that a
planarization approach would not be best suited for our purpose.
From the combinatorial point of view this would imply to deal with
boundaries of arbitrary size (arising from the planarization
procedure), as non-trivial cycles can be of size $\Omega
(\sqrt{n})$, and cut-graphs have size $O(gn)$.
Moreover, from the algorithmic complexity point of view,  the most
efficient procedures for computing small non-trivial
cycles~\cite{CabelloM07,Kutz06} require more than linear time, the
best known bound being currently of $O(n\log n)$ time.

\subsubsection{Schnyder trees and graph encoding}

One of our main motivations for generalizing Schnyder woods to
higher genus is the great number of possible applications in
graph encoding and mesh compression that take advantage of
spanning tree decompositions~\cite{Kee95,Ros99,Tur84}, and in
particular of the ones underlying Schnyder woods (and related
extensions) for planar
graphs~\cite{BarbayISAAC07,Chi01,Chu98,Fus05,He99,Pou03}.
The combinatorial properties of Schnyder woods and the related
characterizations (\emph{canonical orderings}~\cite{Kant96}) for
planar graphs yield efficient procedures for encoding tree
structures based on multiple parenthesis words. In this context a
number of methods have been proposed for the simple
compression~\cite{He99} or the succinct
encoding~\cite{Chu98,Chi01} of several classes of planar graphs.
More recently, this approach based on spanning tree decompositions
has been further extended to design a new succinct encoding of
labeled planar graphs~\cite{BarbayISAAC07}. Once again, the main
ingredient is the definition of three traversal orders on the
vertices of a triangulation, directly based on the properties of
Schnyder woods.
Finally we point out that the existence of \emph{minimal orientations}
(orientations without counterclockwise directed cycles) recently
made it possible to design the first optimal (linear time)
encoding for  triangulations and $3$-connected plane
graphs~\cite{Fus05,Pou03}, based on  bijective
correspondences with families of
plane trees. Such bijective constructions, originally introduced by
Schaeffer~\cite{Schaeffer_thesis},
 have been applied
to many families of plane graphs (also called planar maps)
and give combinatorial interpretations of
enumerative formulas originally found by Tutte~\cite{Tutte63}.
In recent work, some of these bijections are extended to higher
genus~\cite{ChMaSc,Ch},
but a bijective construction for triangulations or 3-connected plane
graphs in higher genus is not yet known.
%
The difficulty of extending combinatorial constructions to higher genus
is due the fact that some fundamental properties, such as the Jordan curve theorem,
hold only in the planar case (genus $0$).
Nevertheless, the topological approach used by \emph{Edgebreaker}
(using at most $3.67$ bits per vertex in the planar case) has been
successfully adapted to deal with triangulated surfaces having
arbitrary topology: orientable manifolds with
handles~\cite{LopesRSST03} and also multiple
boundaries~\cite{LewinerLRV04}. Using a different approach, based
on a partitioning scheme and a multi-level hierarchical
representation~\cite{CastelliWADS05}, it is also possible to
encode a genus $g$ triangulation with
$f$ faces and $n$ vertices using $2.175f+O(g\log f)+o(f)$
bits (or $4.35n+o(gn)$ bits) which is asymptotically optimal for
surfaces with a boundary: nevertheless, the amount of additional
bits hidden in the sub-linear $o(n)$ term can be quite large, of
order $\Theta(\frac{n}{\log n}\log \log n)$.

\subsection{Contributions}
Our contributions start in Section~\ref{sec:higher_genus}, where
 we give a definition of Schnyder woods for triangulations
of arbitrary genus, which extends the definition of
Schnyder for plane triangulations. Then we describe a traversal
algorithm to actually compute such a so-called $g$-Schnyder wood
for any triangulation of genus $g$, in time $O((n+g)g)$. Again our
procedure  extends to any genus the known procedures to
traverse a plane triangulation and compute a Schnyder wood on the
way~\cite{Schnyder89,Brehm_thesis}. Finally, in
Section~\ref{sec:encoding_application}, we show that a
$g$-Schnyder wood yields an algorithm to efficiently encode a
triangulation of genus $g$ and with $n$ vertices, in
$4n+O(g\log(n))$ bits. This is again an extension to arbitrary genus of a
procedure described in~\cite{He99,BB07a} to encode plane triangulations.
Our result matches the same worst-case encoding rate as
\emph{Edgebreaker}~\cite{Ros99}, which uses at most $3.67n$ bits
in the planar case, but requires up to $4n+O(g\log n)$ bits for
meshes with positive genus~\cite{LopesRSST03,LewinerLRV04}. As far
as we know this is the best known rate for linear time (in fixed
genus)
 encoding of triangulations with
positive genus $g$, quite close to the information theory
 bound of $3.24n+\Omega(g\log n)$ bits (a more detailed
discussion is given in
Section~\ref{sec:encoding_application}).

\section{Schnyder woods for Plane Triangulations}
\subsection{Definition}
A \emph{plane triangulation} $\cT$ is a graph with no loops nor
multiple edges and embedded in the plane such
that all faces have degree 3. The edges and vertices of $\cT$
incident to the outer face are called the outer edges and outer vertices.
The other ones are called the inner edges and inner vertices.

We recall here the definition of Schnyder woods for plane
triangulations, which we will later generalize to higher genus.
While the definition is given in terms of local conditions, the
main structural property, as stated in Fact~\ref{fact:trees},
is more global, namely a partition of the
inner edges into 3 trees, see
Figure~\ref{fig:def_realizer}~\footnote{In the figures, the edges of color 0 are solid,
the edges of color 1 are dotted, and the edges of color $2$ are dashed.}.

\begin{definition}[\cite{Schnyder90}]\label{def:Schnyder wood}
Let $\mathcal{T}$ be a plane triangulation, and denote by
$v_0,v_1,v_2$ the outer vertices in counterclockwise (\emph{ccw})
order around the outer face.
A \textbf{Schnyder wood}  of $\mathcal{T}$ is an orientation and
labeling, with labels in $\{0,1,2\}$ of the inner edges of $\cT$  so as
to satisfy the following conditions:
\begin{itemize}
\item \textbf{root-face condition:} for $i\in\{0,1,2\}$, the inner edges incident to the
outer vertex $v_i$ are all ingoing of color $i$.
\item \textbf{local
condition for inner vertices:} For each inner vertex $v$, the
edges incident to $v$ in counterclockwise (\emph{ccw}) order are: one
outgoing edge  colored $2$,
zero or more incoming edges  colored $1$,  one outgoing edge  colored
$0$, zero or more incoming edges  colored $2$, one outgoing edge
 colored $1$, and zero or more incoming edges  colored $0$, which we write concisely as
$$
(\Seq(\dar 1),\uar 0, \Seq(\dar 2),\uar 1,\Seq(\dar 0),\uar 2).
$$
\end{itemize}
\end{definition}

\begin{fact}[\cite{Schnyder90}]\label{fact:trees}
Each plane triangulation $\mathcal{T}$ admits a Schnyder wood.
Given a Schnyder wood on $\mathcal{T}$, the three directed graphs
$T_0$, $T_1$, $T_2$ induced by the edges of  color $0$, $1$, $2$
are trees that span all inner vertices and are naturally rooted at
$v_0$, $v_1$, and $v_2$, respectively.
\end{fact}


\begin{figure}[t]
\centering \scalebox{1}{
\includegraphics{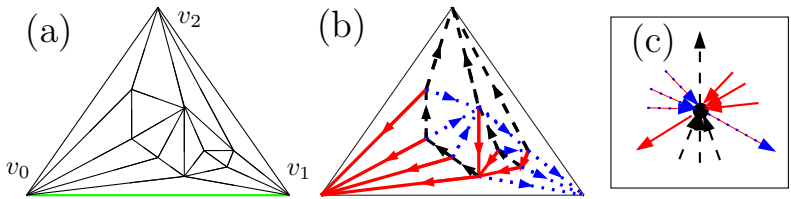}
}
\caption{(a) A rooted planar triangulation, (b) endowed with a Schnyder wood. (c) The
local condition of Schnyder woods.
\label{fig:def_realizer}}
\end{figure}

\subsection{Computation of Schnyder woods for plane triangulations}
\label{sec:computing_planar}

In this section we briefly review a well-known linear time
algorithm designed for computing a Schnyder wood of a plane
triangulation, following the presentation by
Brehm~\cite{Brehm_thesis}.
It is convenient here (in view of the generalization to higher genus)
to consider a plane triangulation as embedded on the sphere $S$, with a marked face
that plays the role of the outer face.
The procedure consists in growing a region $C$, called the
\emph{conquered region}, delimited by a simple cycle $B$ ($B$
is considered as part of $C$)~\footnote{In the figures, the faces of $\cT\backslash C$
are shaded.}. Initially $C$ consists of the
root-face (as well as its incident edges and vertices).
A \emph{chordal edge} is defined as an edge not in $C$ but with
its two extremities on $B$. A free vertex is a vertex of
$B\setminus\{v_0,v_1\}$ with no incident chordal edges. One
defines the \emph{conquest} of such a vertex $v$ as the operation
of transferring to $C$ all faces incident to $v$, as well as the
edges and vertices incident to these faces; the boundary $B$ of
$C$ is easily verified to remain a simple cycle. Associated with a
conquest is a simple rule to color and orient the edges incident
to $v$ in the exterior region. Let $v_r$ be the right neighbor
and $v_l$ the left neighbor of $v$ on $B$, looking toward
$\cT\backslash C$ (in the figures, toward the shaded area). Orient
outward of $v$ the two edges $(v, v_r)$ and $(v, v_l)$; assign
color $0$ to $(v, v_r)$ and color $1$ to $(v, v_l)$. Orient toward
$v$ and color $2$ all edges exterior to $C$ incident to $v$ (these
edges are between $(v, v_r)$ and $(v, v_l)$ in ccw order around
$v$).

The algorithm for computing a Schnyder wood of a plane
triangulation with $n$ vertices is a sequence of $n-2$ conquests
of free vertices, together with the operations of coloring and
orienting the incident edges (the initial conquest, always applied
to the vertex $v_2$, is a bit special:  the edges going to the right and left neighbors are not colored nor oriented, since these are outer edges).

\begin{figure}[th]
\centering \scalebox{1}{
\includegraphics{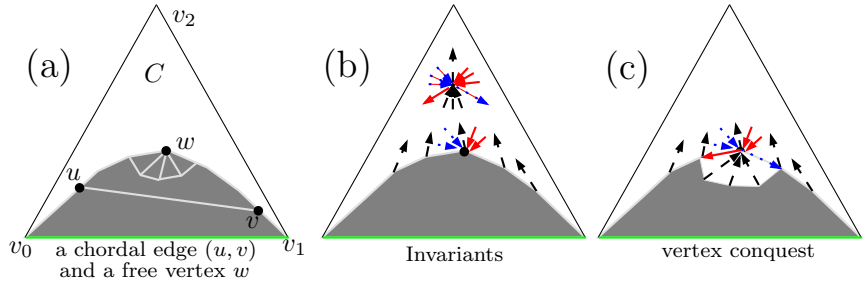}}
\caption{(a) A chordal edge and a free vertex, (b) the invariants valid
in the planar case, (c) the result of a vertex conquest.
\label{fig:def_orientation_planar}}
\end{figure}


The correctness and termination of the traversal algorithm
described above is based on the following
fundamental property illustrated in Figure~\ref{fig:chordal_diag}.
A planar chord diagram  (i.e., a
 topological disk with chordal edges that do not
cross each other) with root-edge $\{v_0,v_1\}$  always has on its
boundary a vertex $v\notin\{v_0,v_1\}$ not incident to any chord,
see for instance~\cite{Brehm_thesis} for a detailed proof.

\begin{figure}
\begin{center}
\scalebox{0.32}{
\includegraphics[width=12cm]{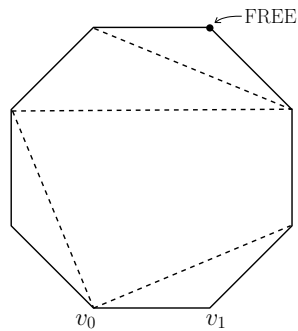}}
\end{center}
\caption{In a planar chord-diagram with a root-edge $e=\{v_0,v_1\}$,
there must be a vertex $v$ not incident to $e$ nor to any chord.}
\label{fig:chordal_diag}
\end{figure}


One proves that the structure computed by the traversal algorithm
is a Schnyder wood by considering some invariants (see
Figure~\ref{fig:def_orientation_planar}):
\begin{itemize}
 \item the edges that are already colored and directed are the inner edges
 of $C\backslash B$.
 \item for each inner vertex $v$ of $C\backslash B$, all edges
incident to $v$ are colored and directed in such a way that the
Schnyder rule (Figure~\ref{fig:def_realizer}(c)) is  satisfied;
 \item every inner vertex $v\in B$ has exactly one outgoing edge $e$
 in $C\backslash B$; and this edge has color $2$.
 Let $v_r$ be the right neighbor and $v_l$ the left neighbor of $v$ on $B$,
 looking toward $\cT\backslash C$. Then all edges strictly between $(v,v_r)$ and
 $e$ in cw order around $v$ are ingoing of color $1$ and all edges strictly between
 $e$ and $(v,v_l)$ in cw order around $v$ are ingoing of color $0$.
\end{itemize}


These invariants are easily checked to be satisfied all along
the procedure (see~\cite{Brehm_thesis} for a detailed presentation), which
yields the following result:

\begin{lemma}[Brehm~\cite{Brehm_thesis}]
Given a planar triangulation $\mathcal{T}$ with outer face $(v_0,
v_1, v_2)$ the traversal algorithm described above
computes  a
Schnyder wood of $\mathcal{T}$ and can be implemented to run in time $O(n)$.
\end{lemma}

Note that a triangulation $\cT$ can have many different Schnyder woods (as
shown by Brehm~\cite{Brehm_thesis}, the set of Schnyder woods of $\cT$ forms
a distributive lattice). Furthermore, the same Schnyder wood can be
obtained from many different total orders on vertices for the above-described
traversal procedure. Such total orders on the vertices of $\cT$ are
called \emph{canonical orderings}~\cite{Kant96}.

\section{Concepts of topological graph theory}
Before generalizing the definition of Schnyder woods and computation methods
to any genus, we need to define the necessary concepts of
topological graph theory. The graphs considered here are
allowed to have loops and multiple edges.

\subsection{Graphs on surfaces, maps, subcomplexes.}\label{sec:subcomplexes}
A \emph{graph on a surface} $M$ is a graph $G=(V,E)$
 embedded without edge-crossings on a closed
orientable surface $S$ (such a surface is specified by its genus
$g$, i.e., the number of handles). If the components of
$S\backslash G$ are homeomorphic to topological disks, then $M$ is
called a (topological) \emph{map}, which implies that $G$ is a
connected graph. A subgraph $G'=(V',E')$ of $G$ is called cellular
if the components of $S\backslash G'$ are homeomorphic to
topological disks, i.e., the graph $G'$ equipped with the
embedding inherited from $G$ is a map. A subgraph $G'=(V',E')$ is
\emph{spanning} if $V'=V$. A \emph{cut-graph} of $M$ is a spanning
cellular subgraph $G'=(V',E')$ with a unique face, i.e.,
$S\backslash G'$ is homeomorphic to a topological disk.

Note that a map has more structure than a graph, since
the edges around each vertex are in a certain cyclic order. In addition, a map has
faces (the components of $M\backslash S$). By the Euler relation, the genus $g$ of
the surface on which $M$ is embedded satisfies
$$
2-2g=\chi(M)=|V|-|E|+|F|,
$$
where $\chi(M)$ is the Euler characteristic of $M$, and $V$, $E$,
and $F$ are the sets of vertices, edges, and faces in $M$. It is
convenient to view each edge $e=\{u,v\}\in E$ as made of two
\emph{brins} (or half-edges), originating respectively at $u$ and
at $v$, the two brins meeting in the middle of $e$; the two brins
of $e$ are said to be \emph{opposite} to each other. (Brins are
also called darts in the literature). The \emph{follower} of a
brin $h$ is the next brin after $h$ in clockwise order (shortly
cw) around the origin $v$ of $h$. A \emph{facial walk} is a cyclic
sequence $(b_1,\ldots,b_k)$, where for $i\in[1..k]$, $b_{i+1}$
(with the convention that $b_{k+1}=b_1$) is the opposite brin of
the follower of $b_i$. A facial walk corresponds to a walk along
the boundary of a face $f$ of $M$ in ccw order (i.e., with the
interior of $f$ on the left).

The face incident to a brin $h$ is defined as the face on the left of $h$ when one looks
toward the origin of $h$.
Note that to a brin $h$ of $M$ corresponds a
\emph{corner} of $M$, which is the pair $c=(h,h')$ where $h'$ is the follower of $h$.
The vertex incident to $c$ is defined as the common origin of $h$ and $h'$,
and the face $f$ incident to $c$ is defined as
the face of $M$ in the sector delimited by $h$ and $h'$
(so $f$ coincides with the face incident to $h$).

Maps can also be defined in a combinatorial way. A
\emph{combinatorial map} $M$ is a connected graph $G=(V,E)$ where
one specifies a cyclic order for the set of brins (half-edges)
around each vertex. One defines facial walks of a combinatorial
map as above (note that the above definition of a facial walk as a
certain cyclic sequence of brins does not need an embedding, it
just requires the cyclic cw order of the brins around each vertex).
One obtains from the combinatorial map a topological map by
attaching a topological disk at each facial walk; and the genus
$g$ of the corresponding surface satisfies again
$2-2g=|V|-|E|+|F|$, with $F$ the number of topological disks
(facial walks), which are the faces of the obtained topological
map~\cite{Mohar_book}.

In this article we will focus on triangulations; precisely a
\emph{triangulation} is a map with no loops nor multiple edges and
with all faces of degree 3 (each face has 3 edges on its contour).

\paragraph*{Duality.}
The \emph{dual} of a (topological) map $M$ is the map $M^*$ on the same surface
defined as follows:
$M^*$ has a vertex in each face of $M$, and each edge $e$ of $M$
gives rise to a \emph{dual edge} $e^*$ in $M^*$, which connects
the vertices of $M^*$ corresponding to the faces of $M$ sharing~$e$.
Note that the adjacencies between the vertices of $M^*$ correspond
to the adjacencies between the faces of $M$. Duality for edges can
be refined into duality for brins: the dual of a brin $h$ of an
edge $e$ is the brin of $e^*$ originating from the face incident to $h$
(the face on the left of $h$ when looking toward the origin of $h$). Note that the
dual of the dual of a brin $h$ is the opposite brin of $h$.

\paragraph*{Subcomplexes.} Given a map $M$ on a surface $S$, with $V$,
$E$, and $F$ the sets of vertices, edges, and faces of $M$, a
\emph{subcomplex} $C=(V',E',F')$ of $M$ is given by subsets
$V'\subset V$, $E'\subset E$, $F'\subset F$ such that the edges
around any face of $F'$ are in $E'$ and the extremities of any
edge in $E'$ are in $V'$. The subcomplex $S$ is called connected
if the graph $G'=(V',E')$ is connected. The \emph{Euler
characteristic} of a connected subcomplex $S$ is defined as
\begin{equation}
\chi(S):=|V'|-|E'|+|F'|.
\end{equation}
%
%
%

\emph{Boundary walks and boundary corners for subcomplexes.} Note
that a connected subcomplex $C$ of $M$ naturally inherits from $M$
the structure of a combinatorial map (the brins for edges in $E'$
inherit a cw cyclic order around each vertex of $V'$). Hence one can
also define facial walks for $C$. Such a facial walk is called a
\emph{boundary walk} for $C$ if it does not correspond to a facial
walk of a face in $F'$. A boundary brin is a brin $h$ in a
boundary walk, and the corresponding \emph{boundary corner} of $C$
$b=(h,h')$ is the pair formed by $h$ and the next brin $h'$
in $C$ in cw order around the origin $v$ of $h$. Note that a
boundary corner of $C$ is not a corner of $M$ if there are brins
$h_1,\ldots,h_k$ of $M\backslash \sin$ in cw order strictly between $h$ and $h'$. These
brins are called the \emph{exterior brins} incident to $b$. By
extension, the edges to which these brins belong are called the
\emph{exterior edges} incident to $b$. The faces of $M$ incident to $v$
in cw order between $h$ and $h'$ are called the \emph{exterior
faces} incident to $b$. Recall that a facial  walk is classically
encoded by the list of brins $(b_1,\ldots,b_k)$, where $b_{i+1}$
is the opposite brin of the follower $b_i'$ of $b_i$ (for a
subcomplex $C$, it means that $b_i'$ is the next brin in $C$ after
$b_i$ in cw order around the origin of $b_i$). For a boundary
walk, one also adds to the list of brins the exterior brins in
each corner, that is, one inserts between $b_i$ and $b_{i+1}$ the
ordered list of brins of $M$ that are strictly between $b_i$ and
$b_i'$ in cw order. The obtained (cyclic) list is called the \emph{complete
list of brins} for the boundary walk. In this list the brins
$b_1,\ldots,b_k$ are called the \emph{boundary brins}, the other
ones are called the exterior brins.

\paragraph*{The topological map associated with a connected subcomplex.}
The topological map $\wh{C}$ associated with $C$ is  obtained
by attaching to each of the $k$ boundary walks a topological disk; therefore
$\chi(\wh{C})=\chi(C)+k$. The genus $g'$ of $\wh{C}$, given by $2-2g'=\chi(\wh{C})$,
is at most the genus $g$ of the surface on which $C$ is embedded.
  The $k$ faces  of $\wh{C}$
corresponding to the added disks are called the \emph{boundary faces} of $\wh{C}$;
 by a slight abuse of terminology, we call these the boundary faces of $C$.
 Note that each boundary walk of $C$ corresponds to a facial walk for a boundary
 face of $\wh{C}$.

\paragraph*{Duality for subcomplexes.}
Given $C=(V',E',F')$ a subcomplex of a map $M$,
the \emph{complementary dual} $D$ of  $C$ is the subcomplex of $M^*$
formed by the vertices of $M^*$ dual to faces in $F\setminus F'$, the edges
of $M^*$ dual to edges in $E\setminus E'$, and the faces of $M^*$ dual to vertices
in $V\setminus V'$.

\begin{lemma}[correspondence between boundary walks]\label{lem:boundary}
Let $C$ be a connected  subcomplex of a map $M$ such that
the complementary dual complex $D$ is also connected.
For a brin $h\in M$ define $\phi(h)=h^*$ if $h\in \sin$
and $\phi(h)=\mathrm{opposite}(h^*)$ if $h\notin \sin$.

If $L=(h_1,\ldots,h_k)$ is the complete  list of brins
of a boundary walk of $C$, then $\Phi(L):=(\phi(h_k),\ldots,\phi(h_1))$ is the
complete list of brins of a boundary walk of $D$.
The exterior brins
of $L$ correspond to the boundary  brins of $\Phi(L)$, and the boundary brins
of $L$ correspond to the exterior brins of $\Phi(L)$.
Since $\Phi$ is involutive,
$\Phi$ induces a bijection between the boundary faces of $C$ and
the boundary faces of $D$.
\end{lemma}

\subsection{Handle operators}\label{sec:handle_operators}
Following the approach suggested in~\cite{LopesRSST03,LewinerLT03}, based on
Handlebody theory for surfaces, we design a new traversal strategy
for higher genus surfaces: as in the planar case, our strategy
consists in conquering the whole  graph incrementally. We use an
operator \texttt{conquer} similar to the conquest of a free vertex used
in the planar case, as well as two new
operators---\texttt{split} and \texttt{merge}---designed to
represent the handle attachments that are necessary in higher genus.
We start by setting some notations and definitions. We consider a genus $g$
triangulation $\cT$ with $n$ vertices.
In addition, we mark an arbitrary face of $\cT$, called
the \emph{root-face}.

The traversal procedure consists in growing a connected
subcomplex of $\cT$, denoted $\sin$, which is initially equal to
the root-face (together with the edges and vertices of the
root-face); and such that the complementary dual subcomplex,
denoted $\sout$, remains connected all along the traversal
procedure.



\subsubsection{Handle operator of first type}\label{sec:conquer}

\begin{definition}
A \emph{chordal edge} is an edge of $\cT\setminus\sin$
whose two brins $h_1$ and $h_2$ are exterior brins of some
boundary corners $b_1$ and $b_2$.
A boundary corner $b$ of $\sin$ is \emph{free} if no exterior edge of $b$ is a
chordal edge.
\end{definition}

We can now introduce the first operator, called $\mathtt{conquer}$
(see Figure~\ref{fig:split_merge_edges}). Given $b$ a free
boundary corner of $\sin$, $\mathtt{conquer}(b)$ consists in
adding to $\sin$ all exterior faces of $\cT$ incident to $b$, as
well as the edges and vertices incident to these faces.

\begin{figure}[th]
\centering \scalebox{1.2}{
\includegraphics{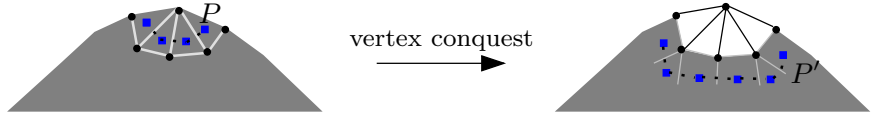}
} \caption{The effect of a conquest on $\sout$ is to delete a set
of vertices $v_1,\ldots,v_r$ together with their incident edges,
denoted by $P=v_1\rightarrow v_2\ldots v_{r-1}\rightarrow v_r$.
Call $D_{\mathrm{old}}$ the complex $D$ before conquest and call
$D_{\mathrm{new}}$ the complex $D$ after conquest. As shown in the
right picture, there is a neighboring path $P'$ disjoint from
$P$. Thanks to $P'$, any path in $D_{\mathrm{old}}$ starting and
ending out of $P$ and passing possibly by vertices and edges of
$P$ can be modified into a path with same starting and ending
vertices but not passing by $P$. Therefore $D_{\mathrm{new}}$ is
connected.} \label{fig:EffectConquest}
\end{figure}

The effect of the conquest on $\sout$ is shown in
Figure~\ref{fig:EffectConquest}; note that
 $D$ remains connected after the conquest.
In addition, the number of boundary faces of $\sin$
is unchanged, as well as the Euler characteristic (indeed, if the number of faces
transferred to $\sin$ is $k$, then the number of vertices transferred
to $\sin$ is $k-1$ and the number of edges transferred
to $\sin$ is $2k-1$). Therefore a conquer operation does
 not modify the topology of
$\sin$.


\subsubsection{Handle operators of second type}

A chordal edge $e$ for $\sin$ is said to be \emph{separating} if
its dual edge $e^*$ is a bridge of $\sout$ (a \emph{bridge} is an
edge whose removal disconnects the graph). Otherwise it is called
non-separating.

\begin{definition}[split edge]
A \emph{split edge} for $\sin$ is a
non-separating chordal edge $e$ such that the two brins of $e$
are incident to boundary corners in the same boundary face of $C$.
\end{definition}
According to the equivalence stated in Lemma~\ref{lem:boundary}, a
split edge $e$ is such that $e^*$ is not a bridge but has the same
boundary face (of $\sout$) on both sides.

We can now define the second operation, \texttt{split}, related to a split
edge $e$: double $e$ into two parallel edges delimiting a face $f$ of degree $2$,
and add  the face $f$ and the two edges representing $e$ to $\sin$.
Note that $\sout$
remains connected since $e^*$ is not a bridge. When doing the split operation,
the boundary walk at the two extremities of $e$ is split into two
boundary walks. Therefore the number of  boundary faces of $\sin$
increases by $1$.
%
%
Note that the Euler characteristic $\chi(\sin)$ decreases by $1$; indeed in $\sin$
the number of vertices is unchanged,
the number of edges increases by $2$ (addition of the split edge, which is doubled)
and the number of
 faces increases by $1$ (addition of the special face). And
 the Euler characteristic of
the map $M$ associated with $\sin$ is unchanged (when including
the boundary faces, the number of
 faces both increases by $2$, as the number of edges), hence the genus of $M$ is
also unchanged.

\begin{definition}[merge edge]
A \emph{merge edge}  for $\sin$ is a
 chordal edge having its two brins incident to boundary corners
in distinct boundary faces of $\sin$.
\end{definition}
According to Lemma~\ref{lem:boundary}, if $e$ is a merge edge,
the faces of $D$ on both sides of $e^*$ are distinct boundary faces, hence
$e^*$ cannot be a bridge of $D$, i.e., $e$ is non-separating.

We can now define the third operation, \texttt{merge}, related to a merge
edge $e$: double $e$ into two parallel edges delimiting a face $f$ of degree $2$,
and add the face $f$ and the two edges representing $e$ to $\sin$.
Note again that $\sout$ remains
connected since $e^*$ is not a bridge. When doing a merge operation, the
boundary faces
 at the two extremities of $e$ are merged into a single boundary face,
so that the number of boundary faces of $\sin$ decreases by $1$.
%
%
Similarly as for a split operation,
 the Euler characteristic $\chi(\sin)$ decreases by $1$ (addition of a doubled special  edge
 and of one special face); and
 the Euler characteristic of
the map $M$ associated with $\sin$ decreases by $2$ (when including
the boundary faces, the number of
 faces is unchanged, and the number of edges increases by $2$),
 hence the genus of $M$ increases by $1$;
informally a merge operation ``adds a handle''.


\begin{figure}[t]
\centering \scalebox{0.9}{
\includegraphics{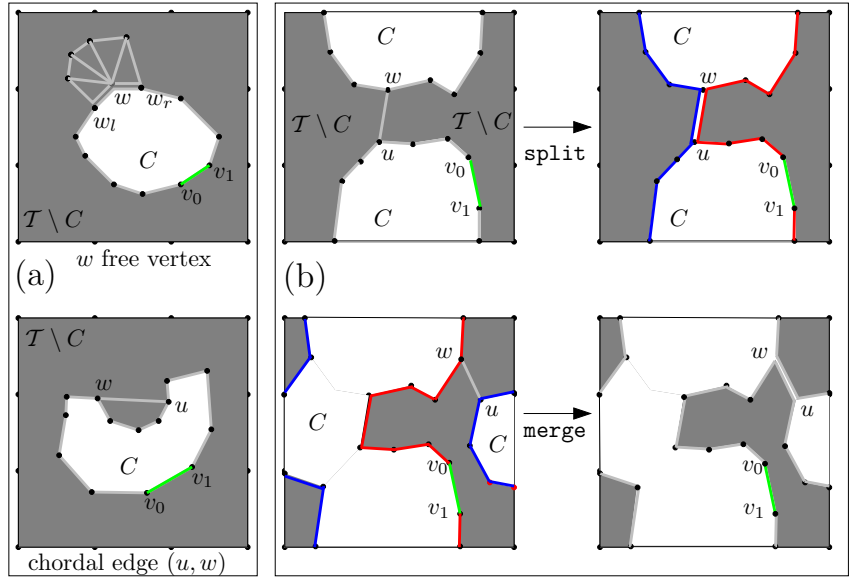}
} \caption{Illustrated on a toroidal graph, (a) the result of a
\texttt{conquer} operation, and a contractible chordal edge
$(u,w)$ (in gray); (b) the result of a \texttt{split} (respectively, \texttt{merge})
operation on a split edge $(u,w)$ (respectively, merge edge $(u,w)$).
\label{fig:split_merge_edges}}
\end{figure}


\section{Schnyder woods for triangulations of arbitrary genus}\label{sec:higher_genus}

\subsection{Definition of Schnyder Woods extended to arbitrary genus}\label{sec:g_Schnyder wood}

We give here a definition of Schnyder woods for
triangulations  that extends to arbitrary genus the definition known in
the planar case, see Figure~\ref{fig:def_grealizer} for an
example. We consider here triangulations of genus $g$ with a marked face, called
the root-face. As in the planar case, the edges and vertices are called outer
or inner whether they are incident to the root-face or not.

\begin{definition}\label{def:g_Schnyder wood}
Consider a genus $g$ triangulation $\cT$ with
$n$ vertices, and having a root-face
$f=(v_0, v_1, v_2)$ (the vertices are ordered according to a
walk along $f$ with the interior of $f$ on the right). Let $\cE$ be
the set of inner edges of $\cT$.
A \emph{$g$-Schnyder wood} of $\cT$ (also called genus $g$ Schnyder wood)
is a partition of
$\cE$  into a set of \emph{normal edges}  and a set
$\mathcal{E}^s$ of
\emph{special edges} considered as fat, i.e., each special edge is
 doubled into two edges delimiting a face of degree $2$, called a \emph{special face}.
 In addition, each edge, a normal edge or one of the two
edges of a special edge, is directed and has a
label (also called color) in $\{0,1,2\}$, so as to satisfy the following conditions:
\begin{itemize}

\item \textbf{root-face condition:}
The outer vertex $v_2$ is incident to no special edges. All inner edges
incident to $v_2$ are ingoing of color $2$.

Let $k\geq 0$ be the number of special edges incident to $v_0$ (each of these
special edges is doubled), and let $L=(e_1,f_1,e_2,f_2,\ldots,e_r,f_r)$
be the cyclic list of edges and faces incident to $v_0$ in ccw order ($f_i$
is the face incident to $v_0$ between $e_i$ and $e_{(i+1)\ \mathrm{mod}\ r}$).
A \emph{sector} of $v$ is a maximal interval of $L$ that does not contain
a special face nor the root-face. Note that there are $k+1$ sectors, which are
disjoint; the one containing the edge $\{v_0,v_1\}$ is called the \emph{root-sector}.

Then, all inner edges in the root-sector are ingoing of color $0$.
In all the other $k$ sectors, the edges in ccw order are of the form
$$
\Seq(\dar 1),\uar 0, \Seq(\dar 2),\uar 1,\Seq(\dar 0).
$$
The definitions of sectors and conditions are the same for $v_1$,
except that all edges in the root-sector are ingoing of color $1$.

\item \textbf{local condition for inner vertices:} Every
inner vertex $v$ has exactly one outgoing edge $e$ of color $2$.
Let $k$ be the number of special edges incident to $v$ (each of these
edges is doubled and delimits a special face), and
let $L=(e_1,f_1,e_2,f_2,\ldots,e_r,f_r)$
be the cyclic list of edges and faces incident to $v$ in ccw order.
A \emph{sector} of $v$ is a maximal interval of $L$ that does not contain
a special face nor the edge $e$. Note that there are $k+1$ sectors around $v$,
which are disjoint.

Then, in each sector the edges in ccw order are of the form
$$
\Seq(\dar 1),\uar 0, \Seq(\dar 2),\uar 1,\Seq(\dar 0).
$$

\item \textbf{Cut-graph condition:} The graph $T_2$ formed by the edges of color 2
is a tree spanning all vertices except $v_0$ and $v_1$, and rooted at $v_2$,
 i.e.,  all
edges of $T_2$ are directed toward $v_2$. The embedded
subgraph $G_2$ formed by $T_2$ plus the two edges $(v_0,v_2)$ and
$(v_1,v_2)$
plus the special edges (not considered as doubled here) is a cut-graph of $\cT$,
which is called the \emph{cut-graph} of the Schnyder wood.
\end{itemize}
(Note that the cut-graph condition forces the number of special edges to be $2g$.)
\end{definition}


\begin{figure}[t]
\centering \scalebox{0.95}{
\scalebox{1.4}{\includegraphics{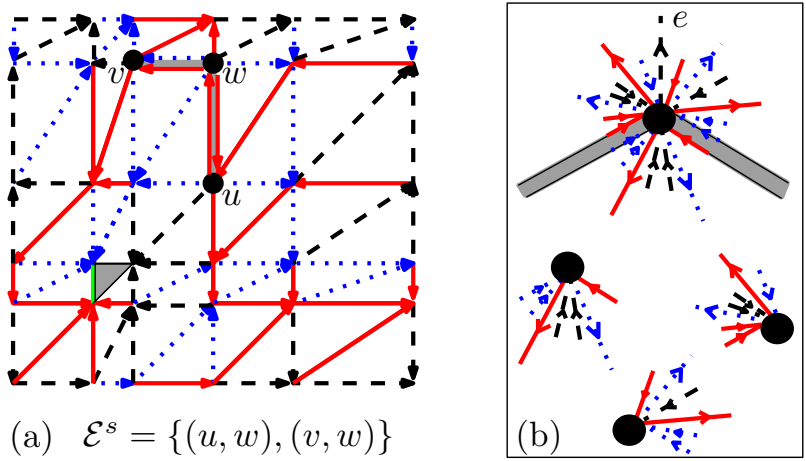} } }
\caption{(a) A toroidal triangulation endowed with a $g$-Schnyder
wood (the root-face is dashed). (b) The local condition for an
inner vertex with two special edges (each of which is doubled and
delimits a 2-sided face), below are shown the 3 sectors delimited
by the special edges and the outgoing edge of color $2$.
\label{fig:def_grealizer}}
\end{figure}


As an example, Figure~\ref{fig:def_grealizer}(a) shows a toroidal
triangulation endowed with a g-Schnyder wood.

\emph{Remark~1}.
Note that if an inner vertex $v$ is incident to no special edge, then there is
a unique sector around $v$, which is formed by all edges incident to $v$
except the outgoing one of color $2$. The local condition above
implies that the edges around $v$ are of the form
$$
(\Seq(\dar 1),\uar 0, \Seq(\dar 2),\uar 1,\Seq(\dar 0),\uar 2),
$$
as in the planar case. Since at most $4g$ vertices are incident to special edges,
our definition implies that in fixed genus, almost all inner vertices satisfy
the same local condition as in the planar case. In addition the vertices
incident to special edges satisfy a local condition very similar to the one
in the planar case.

\paragraph*{Remark~2.} The last condition, stating that $T_2$
is a tree,  is redundant in the planar case
(it is implied by the local
conditions) but not in higher genus: one easily finds an example
of structure where all local conditions are satisfied but the edges of color $2$
form many disjoint circuits.

\paragraph*{Remark~3.}
Finally, we point out (see Proposition~\ref{theo:maps} and the remark
after) that $g$-Schnyder woods (precisely, those
computed by a traversal algorithm described later on) give rise to
 decompositions into 3 spanning cellular subgraphs, one with one face
 and the two other ones with $1+2g$ faces.
 This generalizes the decomposition of a plane triangulation
 into 3 spanning trees.


\subsection{Computing Schnyder woods for any genus}

This section presents  an algorithm for traversing a triangulation
of arbitrary genus $g\geq 0$ and computing a $g$-Schnyder wood on
the way. Our algorithm naturally extends to any genus the
procedure of Brehm. As in the planar case, the traversal is a
greedy sequence of conquest operations, with here the important
difference that these operations are interleaved with $2g$
merge/split operations. Another point is that, in higher genus,
the region that is grown is more involved than in the planar case
(recall that in the planar case, the grown region is delimited by a simple cycle).
This is why we need
the more technical terminology of  subcomplex. It also turns out
that a vertex might appear several times on the boundary of the
grown complex, therefore we have to use the refined notion of free
boundary corner, instead of free vertex in the planar case (in the
planar case, a vertex appears just once on the boundary of the
grown region).

Let us now give the precise description of the traversal procedure on
a triangulation of genus $g$ with a root-face.
As in the planar case, we grow a ``region'' $\sin$. Precisely, $\sin$ is a connected
subcomplex all along the traversal.
Initially, $\sin$ is the root-face $\{v_0, v_1, v_2\}$,
together with the edges and vertices of that face; at the end,
$\sin$ is equal to $\cT$. We make use of the operation
\texttt{conquer}$(b)$---with $b$ a free boundary corner of
$\sin$---as defined in Section~\ref{sec:conquer}. Associated with
such a conquest is the \texttt{colorient} rule, similar to the
operation for free vertices described in
Section~\ref{sec:computing_planar} (planar case):

\paragraph*{colorient}
\texttt{colorient}($b$), with $b$ a free boundary
corner of $\sin$: let $v$ be the vertex incident to $b$, and let
$e$, $e'$ be the two edges delimiting $b$, with $e'$ after $e$ in cw order around $v$. Orient $e$ and $e'$
outward of $v$, giving color $1$ to $e$ and color $0$ to $e'$.
Orient all the exterior edges of $b$ toward $v$ and give color $2$
to these edges (these edges are strictly between $e$ and $e'$ in cw order around $v$).

We also make use of the handle operations \texttt{split} and
\texttt{merge}, as defined in Section~\ref{sec:handle_operators}.
Define an \emph{update-candidate} for $C$ as either a free
boundary corner, or a split edge, or a merge edge.

\smallskip
\noindent \textsc{ComputeSchnyderAnyGenus($\cT$)} ($\cT$ a triangulation of genus $g$)\\
Initialize $\sin$ as the root-face $f$ plus the vertices and edges of $f$; \\
\noindent \texttt{while $\sin\neq\cT$ find an update-candidate $\sigma$ for $\sin$} \\
\indent \texttt{If $\sigma$ is a free boundary corner $b$} \\
\indent \indent \texttt{conquer($b$);}
\indent\indent \texttt{colorient($b$);} \\
\indent \texttt{If $\sigma$ is a merge edge $e=\{u,w\}$ for $\sin$} \\
\indent \indent \texttt{merge(u,w)}; \\
\indent \texttt{If $\sigma$ is a split edge $e=\{u,w\}$ for $\sin$} \\
\indent \indent \texttt{split(u,w)}; \\
\noindent \texttt{end while} \\

Note that the above algorithm performs conquests, merge operations, and split
operations in whichever order, i.e., with no priority on the 3 types of operations.

Figure~\ref{fig:realizer_on_torus} shows the traversal algorithm
executed on a toroidal triangulation. Observe the subtlety that,
for positive genus, the vertices incident to merge/split edges
have several corners that are conquered, as illustrated in
Figure~\ref{fig:def_orientation}. Precisely, for a vertex $v$
incident to $k\geq 0$ merge/split edges, its conquest occurs $k+1$
times if $v$ is an inner vertex and $k$ times if $v\in\{v_0,v_1\}$.

Note also that, if the algorithm terminates (which will be proved
next), the number of merge edges must be $g$ and the number of
split edges must be $g$. Indeed, in the initial step, $\sin$ has
$k=1$ boundary face and genus $g'=0$, while (just before) the last
step $\sin$ has $k=1$ boundary face and genus $g'=g$. Since the
effect of each split is $\{k\leftarrow k+1,g'\leftarrow g'\}$ and
the effect of each merge is $\{k\leftarrow k-1,g'\leftarrow
g'+1\}$, there must be the same number of splits as merges (for
$k$ to be the same finally as initially) and the number of merges
must be $g$ (for $g'$ to increase from $0$ to $g$). As we will
see, these $2g$ edges  are the special edges of the Schnyder wood
computed by the traversal algorithm.

\begin{figure}[th]
\centering \scalebox{1.0}{
\includegraphics{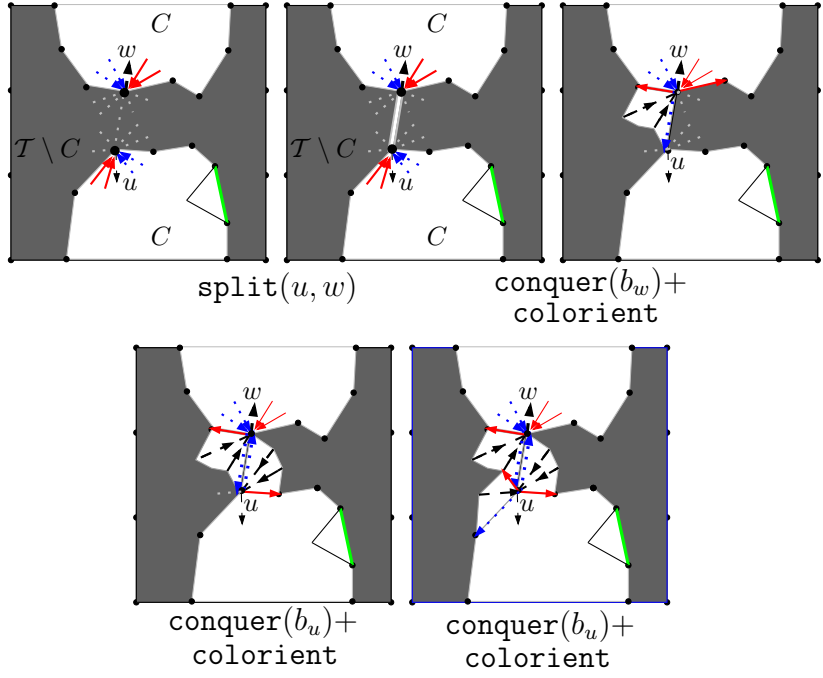}}
\caption{These pictures show the result of \texttt{colorient}
operations in the higher genus case. Any split (or merge) edge
$(u,w)$ can be directed in one or two directions (having possibly
two colors), depending on the traversal order on its extremities
(we denote by $b_w$ a boundary corner incident to vertex $w$).
\label{fig:def_orientation}}
\end{figure}

\begin{figure}[t]
\centering \scalebox{0.7}{
\includegraphics{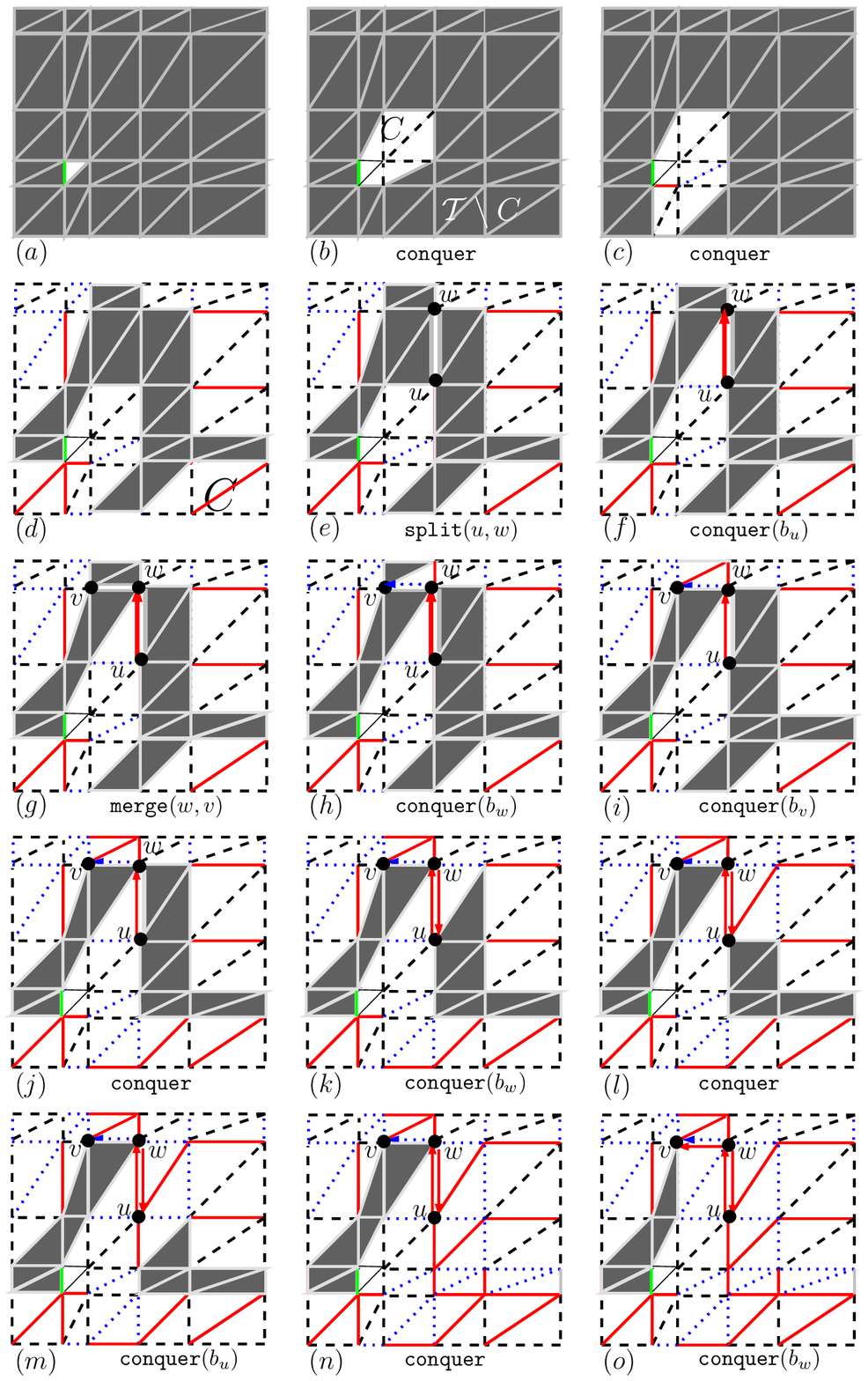}
} \caption{Execution of our traversal algorithm. (a) The traversal
starts with a conquest at the outer vertex $v_2$.
(b)-(c) As far as only $\mathtt{conquer}$ operations, (d) the area
already explored (white triangles) remains homeomorphic to a disk.
Whenever there remain no free corners, it is possible to find
\emph{split} (e) and \emph{merge} (g) edges (incident to black
circles).
%
Once the region $\cT \setminus C$ is a topological disk (h), the
traversal can be completed with a sequence of \texttt{conquer}
operations.
%
 \label{fig:realizer_on_torus}}
\end{figure}

\begin{theorem}
Any triangulation $\cT$ of genus $g$ admits a $g$-Schnyder wood,
which can be computed in time $O((n+g)g)$.
\end{theorem}

This theorem is proved in several steps: first we show in
Lemma~\ref{lemma:termination}  that the traversal algorithm
terminates and in Lemma~\ref{lem:time_exec} that it can be implemented to run in time
$O((n+g)g)$. Then we show in Lemma~\ref{lem:local} (local conditions) and
 Corollary~\ref{cor:cut} (cut-graph
condition) that it
computes a $g$-Schnyder wood.

\subsection{Termination and complexity of the algorithm}
Here $\sin$ denotes the growing subcomplex in the traversal algorithm, and
$\sout$ denotes the complementary dual of $\sin$.

\begin{lemma}[Termination]\label{lemma:termination}
Let $\cT$ be a genus $g$  triangulation. Then at any step of \textsc{ComputeSchnyderAnyGenus}($\cT$) strictly
before termination, there is an update-candidate incident to the boundary face containing $\{v_0,v_1\}$.
Hence the procedure \textsc{ComputeSchnyderAnyGenus}($\cT$)
 terminates.
\end{lemma}
\begin{proof}
Consider the boundary face $f_0$ of $C$ containing the edge
$\{v_0,v_1\}$, at some step strictly before termination of the traversal.
Assume that there is no split edge nor merge edge incident to $f_0$ (i.e.,
no split nor merge edge has one of its two extremities incident to a boundary
corner of $f_0$):
we are going to show that, in this
case, there must be a free boundary corner incident to $f_0$.
Each chordal edge $e$ incident to $f_0$ is separating. Hence
$e$ is in fact incident to $f_0$ at its
two extremities (otherwise $e$ would be a merge edge). Consider
the complete list $L$ of brins around $f_0$, as defined in
Section~\ref{sec:subcomplexes}. Let $d$ and $e$ be any pair of
chordal edges incident to $f_0$ (provided $f_0$ has at least two
incident chordal edges). Note that $d^*$ and $e^*$ are bridges of $\sout$.

We claim that the brins $(d_1,d_2)$ of $d$ and $(e_1,e_2)$ of $e$ are not
in a crossing-configuration, i.e., cannot appear as $(\ldots,d_1,\ldots,e_1,\ldots,d_2,\ldots,e_2,\ldots)$ in $L$. Indeed, if the order was so, Lemma~\ref{lem:boundary}
would imply
that the dual brins appear as $(\ldots,e_2^*,\ldots,d_2^*,\ldots,e_1^*,\ldots,d_1^*,\ldots)$
in $\Phi(L)$. But this would imply that the dual edge $d^*$ of $d$
belongs simultaneously to the two connected components of $D\backslash e^*$.

Hence the cyclic boundary of $f_0$ (the contour of $f_0$ unfolded as a cycle)
together with its chordal edges forms a planar chord-diagram with a root-edge
$\{v_0,v_1\}$, as shown in Figure~\ref{fig:chordal_diag}.
It is well known that, in such a
diagram (as shown for instance by Brehm~\cite{Brehm_thesis}),
one can find a vertex $v\notin\{v_0,v_1\}$ not incident to any chord.
The corner at that vertex is hence free.
\end{proof}


\begin{lemma}[Execution time]\label{lem:time_exec}
The algorithm
\textsc{ComputeSchnyderAnyGenus}($\cT$) can be implemented 
to have running time $O((n+g)g)$---with $g$ the genus and
$n$ the number of vertices of $\cT$---and such that the update-candidate
is always incident to the boundary face containing $\{v_0,v_1\}$.
\end{lemma}
\begin{proof}
At each step, call $f_0$ the boundary face of $\sin$
containing $\{v_0,v_1\}$ and call $f_0^*$ the corresponding boundary face of $\sout$.
Note that there are $2g$ merge/split operations during the execution of the
algorithm. Accordingly, the execution time consists of $2g+1$ \emph{periods}:
each of the $2g$ first periods ends with a merge/split, and the last period finishes
the traversal. To prove that the execution time is $O((n+g)g)$, it is enough to show that each period can be implemented to run in
time $O(|E|)$, with $|E|$ the number edges of the triangulation
(by the Euler relation, $|E|$ is $O(n+g)$).
Our implementation here chooses always an update-candidate
incident to $f_0$ and gives priority to free boundary corners over split and merge edges.

We manipulate maps using the half-edge data-structure; each brin has several pointers:
to the incident vertex, the incident face, the opposite brin, the following brin,
and the dual brin.
There are fixed half-edge data structures
 for the triangulation $\cT$ and for its dual $\cT^*$,
and there are evolving half-edge data-structures for $\sin$ and
for the complementary dual $\sout$. Each brin of $\sout$ incident
to a boundary face is dual to a brin exterior to a boundary corner
of $\sin$. Accordingly such a brin of $\sout$ has an additional
pointer to the corresponding boundary corner of $\sin$ (a boundary
corner of $\sin$ is identified with a boundary brin of $\sin$) .
And the brins of $\sout$ that are on an edge with a boundary face
on both sides have a flag indicating this property; the dual of
these edges are precisely the chordal edges for $\sin$. The
boundary corners of $\sin$ have an additional parameter indicating
the number of incident chordal edges. Hence, those that have this
parameter equal to $0$ are the free boundary corners (except for
the two corners at each extremity of $\{v_0,v_1\}$). The free boundary corners
 incident to $f_0$ are
stored in a list. As long as this list
is not empty, one chooses the free boundary corner at the head of the list and performs
the conquest/colorient operations.
%
After performing a conquest, as shown in Figure~\ref{fig:EffectConquest},
some edges of $\sout$ are deleted and some faces $f_1,\ldots,f_r$
of $\sout$ are merged with a boundary face of $\sout$. The edges
of $f_1,\ldots,f_r$ that are not deleted are called
\emph{uncovered} by the conquest. Note that the only edges that
might change status (i.e., become chordal) are the uncovered
edges. If an uncovered edge $e$ becomes chordal (i.e., has now a
boundary face of $\sout$ on both sides),
 one updates the status of $e$ as chordal,
and accordingly one increments the parameter for the number of incident chordal edges
 of the boundary corners (for $\sin$) at the two extremities of the dual edge of $e$.
Since an edge can be uncovered by at most two conquests and since the number
of operations performed on an uncovered edge is constant, the complexity of
updating the half-edge data structures over the whole period is $O(|E|)$.

At the end of a period, there is no free boundary corner incident to $f_0$.
Hence, by Lemma~\ref{lemma:termination},
either the algorithm directly terminates, or there is a merge or split edge
incident to $f_0$.
To check for a merge edge incident to $f_0$, one
scans the edges of $\sout$. If there is an edge $e\in\sout$ having distinct boundary faces
on both sides and one of these faces is $f_0^*$, then one
performs a merge operation at $e$, which finishes the period.
Note that scanning all edges of $\sout$ in search of merge edges takes
time $O(|E|)$.

If the traversal is not finished and one finds no merge edge
incident to $f_0$, then by Lemma~\ref{lemma:termination} there
must be a split edge incident to $f_0$, i.e., an edge of $\sout$
that is not a bridge but has $f_0^*$ on both sides. One can find
all the bridges of $\sout$ in $O(|E|)$ time using the depth-first
search principles of Tarjan~\cite{Tarjan72,Tarjan74}. Then one
looks for a non-bridge edge $e$ of $\sout$ with $f_0^*$ on both
sides, and performs a split operation at $e$, which finishes the
period. Again this scanning process in search of a split edge
takes time $O(|E|)$. \end{proof}


\subsection{The local conditions}
We introduce some invariants on the colors
and directions of the edges of a genus $g$ triangulation $\cT$
 that remain satisfied all along the
traversal and ensure that the computed structure is a $g$-Schnyder
wood.

In order to describe the invariants, we need to introduce some terminology.
First we recall that the special edges are ``fat'', i.e., considered as two parallel
edges that delimit a face of degree 2 (this face is part of $\sin$ as soon as
the special edge is in $\sin$).
Given a vertex $v\in\sin$, let $L=(e_1,f_1,e_2,f_2,\ldots,e_r,f_r)$ be the sequence
of edges and faces (which are either triangular or special)
incident to $v$ in ccw order around $v$. In this list, the faces
that are special (2-sided) are only those for special edges that are \emph{already}
in $\sin$. Let us first introduce two invariants that are easily checked
to remain satisfied all along the traversal:
\begin{itemize}
\item
The edges already colored and directed are those whose two incident faces are
in $\sin$ (we include the special faces for the special edges already in $\sin$).
\item
Each inner vertex $v\in\sin$ has a unique outgoing edge of color $2$;
the outer vertices do not have any outgoing edge
of color $2$.
\end{itemize}

At each step, let $k$ be the number of special edges of $\sin$ incident to $v\in\sin$.
If $v$ is an inner vertex of $\cT$, define a sector as a maximal interval of $L$ that
contains no special face nor the outgoing edge $e$ of color $2$. Note that
$v$ has $k+1$ sectors, which are disjoint. A sector is called filled if all its faces
are in $\sin$. We introduce the following invariants:
\begin{itemize}
\item
Both faces incident to $e$ are in $\sin$.
\item
The edges in each filled sector are in ccw order:
$$
\Seq(\dar 1),\uar 0, \Seq(\dar 2),\uar 1,\Seq(\dar 0).
$$
\item
In each non-filled sector the faces not in $\sin$ form an interval $I$ of faces around $v$.
In ccw order in the sector, the directed/colored
edges of $\sin$ before $I$ are ingoing of color $1$,
and the directed/colored edges of $\sin$ after $I$ are ingoing of color $0$.
\end{itemize}

Similarly we define an invariant for $v_2$ (which is true from the first
conquest):
\begin{itemize}
\item
All inner edges incident to $v_2$ are non-special and are ingoing of color $2$.
\end{itemize}

Finally we define invariants for $v_0$ (and similarly for $v_1$).
At each step, let $k$ be the number of special edges of $\sin$ that are incident to $v_0$.
Let $L=(e_1,f_1,e_2,f_2,\ldots,e_r,f_r)$ be the sequence
of edges and faces (which are triangular or special)
incident to $v_0$ in ccw order around $v_0$ (again, the special
faces are those for special edges already in $\sin$).
Define a sector as a maximal interval of $L$ that
contains no special face nor the root-face. Note that
$v$ has $k+1$ sectors, which are disjoint; the one containing the edge $\{v_0,v_1\}$
is called the root-sector. Again a sector is called filled if all its faces
are in $\sin$. We introduce the following invariants:
\begin{itemize}
\item
In each sector the faces not in $\sin$ form an interval $I$ of faces around $v_0$.
\item
The non-root face incident to $\{v_0,v_1\}$ is never in $\sin$ strictly before
termination.
Hence the root-sector is never filled strictly before termination. All the colored/directed edges in the root-sector
are going toward $v_0$ and have color $0$.
\item
The edges in each filled non-root sector are in ccw order:
$$
\Seq(\dar 1),\uar 0, \Seq(\dar 2),\uar 1,\Seq(\dar 0).
$$
\item
In ccw order in a non-filled non-root sector, the directed/colored
edges of $\sin$ before $I$ are ingoing of color $1$,
and the directed/colored edges of $\sin$ after $I$ are ingoing of color $0$.
\end{itemize}
\noindent The invariants are the same for $v_1$, except that the colored/directed edges in the root-sector
are going toward $v_1$ and have color $1$.

One easily checks that these invariants remain satisfied after each conquest, split,
or merge operation.

\begin{lemma}\label{lem:local}
The structure computed by \textsc{ComputeSchnyderAnyGenus}($\cT$)
satisfies the local conditions of a $g$-Schnyder wood.
\end{lemma}
\begin{proof}
At the end, the fact that the invariants are satisfied directly
implies that the local conditions for edge directions and colors
of a $g$-Schnyder wood are satisfied.
\end{proof}


\subsection{The cut-graph property.}
Let $\cT$ be a genus $g$ triangulation on which the traversal
algorithm is applied. Let $G_2$ be the graph formed by the edges
of color $2$, the two edges $\{v_1,v_2\}$ and
$\{v_0,v_2\}$, and the $2g$ special edges, not considered as doubled here.

\begin{lemma}\label{lem:cutGraph}
At each step strictly before the end of the traversal algorithm, let $\Msin$ be the
map associated with $\sin$ and let $\oC$ be the embedded subgraph
of $G_2$ consisting of the edges and vertices of $G_2$ that are in $\sin$.

Then $\oC$ is a cellular spanning subgraph of $\Msin$. In addition there
is a natural bijection between the faces of $\oC$ and the boundary faces of $\Msin$:
each boundary face of $\Msin$ is included in a unique face of $\oC$.\end{lemma}
\begin{proof}
First let us observe that $\oC$ is a cellular spanning subgraph of $\Msin$
iff it is connected, spanning, and has the same genus as $\Msin$.

The property is true initially. Indeed, $\sin$ is the root-face,
which is planar, so $\Msin$ is the triangulation of the sphere with one inner face and
one root-face, which plays the role of the boundary face;
whereas $\oC$ consists of the two edges $\{v_1,v_2\}$
and $\{v_0,v_2\}$, so $\oC$ is a spanning tree of $\Msin$.

Let $k$ be the number of boundary faces of $\Msin$,
which is also the number of faces of $\oC$, and let $g'$ be the
common genus of $\Msin$ and $\oC$ before an operation is performed.
Let us prove that the property stated in the lemma remains true after the
operation, whether a conquest (except the last conquest), a merge, or a split.

Consider a conquest of a free boundary corner $b$, strictly before
the very last conquest (which closes $\sin$).
The new vertices appearing in $\sin$ are connected to the former graph $\oC$
by an outgoing edge of color $2$ in the new graph $\oC$,
hence $\oC$ is still a connected spanning subgraph of $\sin$
after the conquest.  Note also that the genera of $\Msin$ and $\oC$ are unchanged (these two numbers stay equal to $g'$).
Similarly the number of boundary faces of $\Msin$ and the number
of faces of $\oC$ are
 unchanged (these two numbers stay equal to $k$). Finally, as shown
 in Figure~\ref{fig:ProofLem_conquest}, the boundary face of $\Msin$
incident to $b$ is still contained in the corresponding
 face of $\oC$ after the conquest. Hence the property stated in the lemma
 remains true after a conquest.

Now let us consider a split operation. The new split edge
``splits'' a boundary face of $\Msin$ into two faces $f_1$ and
$f_2$, and in the same way splits the corresponding face of $\oC$
into two faces $f_1'$ and $f_2'$ such that $f_1'$
contains $f_1$ and $f_2'$ contains $f_2$. Thus the correspondence
between boundary faces of $\Msin$ and faces of $\oC$ remains true.
In addition, the genera of $\Msin$ and of $\oC$ remain unchanged, equal to $g'$,
hence $\oC$ remains a cellular subgraph of $\Msin$, and is still spanning
(no vertex is added to $\Msin$ nor to $\oC$).
Hence the property remains true
after a split.

Finally consider a merge. The new merge edge ``merges'' two
boundary faces $f_1$ and $f_2$ of $\Msin$ into a single face,
thereby adding a handle (informally, the handle serves to
establish a bridge so as to connect and merge the two faces).
Doing this the two corresponding faces $f_1'$ and $f_2'$ of $\oC$
are also merged into a single face that contains the merger of
$f_1$ and $f_2$, see Figure~\ref{fig:TwoCycles}. Thus the correspondence
between boundary faces of $\Msin$ and faces of $\oC$ remains true.
In addition, the genera of $\Msin$ and of $\oC$ both increase by $1$,
they are equal to $g'+1$ after the merge, so $\oC$ remains a cellular
subgraph of $\Msin$, and is still spanning
(no vertex is added to $\Msin$ nor to $\oC$).
Hence the property remains true
after a merge.
\end{proof}

\begin{corollary}\label{cor:cut}
The graph $G_2$ is a cut-graph of $\cT$.
\end{corollary}
\begin{proof}
Before the very last conquest, $\oC$ becomes equal to $G_2$; and
$\sin$ is equal to $\cT$ minus the triangular face $f$ on the
other side of the root-face from the base-edge $\{v_0,v_1\}$.
Hence the map $\Msin$ associated with $\sin$ is equal to $\cT$, up
to marking $f$ as a boundary face. According to
Lemma~\ref{lem:cutGraph}, $\oC=G_2$ is a spanning cellular
subgraph of $\Msin=\cT$ and has a unique face (since $M$ has a
unique boundary face), hence $G_2$ is a cut-graph of $\cT$.
\end{proof}

\begin{figure}[th]
\centering \scalebox{1}{
\includegraphics{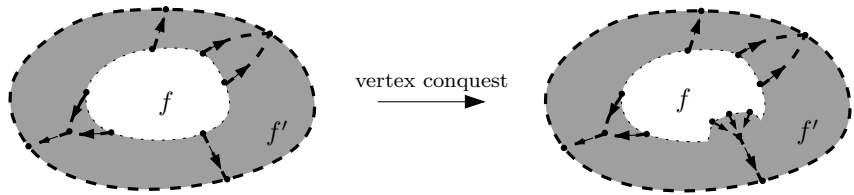}
} \caption{A conquest of a free boundary corner shrinks the
interior of a boundary face $f$ (contour in dotted lines) as well
as the interior of the face $f'$ (contour in dashed lines) of
$\oC$ that contains $f$ (for the sake of clarity, the faces of
$\sin$ are shaded in this figure). The inclusion $f\subset f'$
remains true after the conquest. } \label{fig:ProofLem_conquest}
\end{figure}

\begin{figure}[th]
\centering \scalebox{0.60}{
\includegraphics{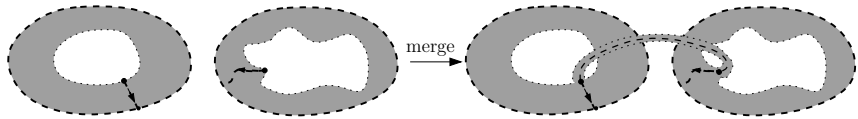}}
\caption{The effect of a merge operation on the growing subcomplex
$\sin$ and on $\oC$ (the faces of $\sin$ are shaded in this
figure). Two faces of $\oC$ are merged and the two corresponding
boundary faces of $\sin$ are merged (the contours of the boundary
faces of $\sin$ are dotted while the contours of the faces of
$\oC$ are dashed).} \label{fig:TwoCycles}
\end{figure}


\subsection{The graphs in color $0$ and $1$ are also cellular}

In this section we show that a $g$-Schnyder wood
computed by the traversal algorithm yields a decomposition of a
triangulation into 3 spanning cellular subgraphs $G_0$, $G_1$, $G_2$,
with $G_2$ having one face ($G_2$ is the cut-graph of the Schnyder wood) and $G_0$
and $G_1$ having each $1+2g$ faces.
This is a natural extension of the property that a planar Schnyder
wood yields a decomposition of a plane triangulation into 3
spanning trees.

\begin{proposition}\label{theo:maps}
Let $\cT$ be a triangulation of genus $g$ endowed with a
$g$-Schnyder wood computed by the algorithm
\textsc{ComputeSchnyderAnyGenus}. The special edges are doubled
(thus $\cT$ gets $2g$
additional degenerated faces of degree 2).

Let $G_0$  be the graph formed by the edges with color $0$
 plus the outer edges
incident to $v_0$. Then $G_0$ is
a spanning cellular subgraph of $\cT$ with $1+2g$ faces (where
some of the  faces might be degenerated, of degree
2). Similarly the graph $G_1$ formed by the edges of color $1$ plus the
two outer edges incident to $v_1$ is a spanning cellular subgraph of $\cT$
with $1+2g$ faces.
\end{proposition}
\begin{figure}
\begin{center}
 \scalebox{0.98}{
 \includegraphics{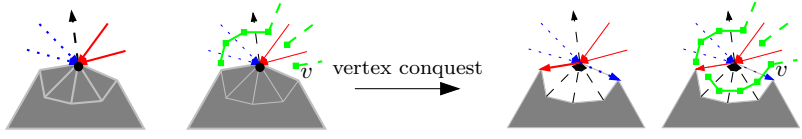}
 }
\end{center}
\caption{The effect of a conquest on the complementary dual $D_0$
of $G_0$ is to attach a chain at a vertex $v$, hence $D_0$ remains acyclic.} \label{fig:dual_conq}
\end{figure}
\begin{proof}
By the local conditions of $g$-Schnyder woods, $G_0$ spans all
inner vertices  (each such vertex is
incident to at least one edge of color $0$). Since one adds the
two edges $\{v_0,v_2\}$ and $\{v_0,v_1\}$, $G_0$ also
spans the vertices of the root-face, so $G_0$ is a spanning
subgraph of $\cT$. Let $\cT^*$ be the dual map of $\cT$. To show
that $G_0$ is cellular, it is enough to show that the
complementary dual $D_0$ of $G_0$ is acyclic ($D_0$ is the
subgraph of $\cT^*$ induced by all vertices of $\cT^*$ and by the
edges of $\cT^*$ that are dual to the edges of $\cT\setminus
G_0$).
At each step of the traversal algorithm, let $D_0'$ be the
subgraph of $D_0$ induced by the edges of $D_0$ dual to edges
having a face in $\sin$ on both sides.
Let us show that $D_0'$ remains acyclic (i.e., a forest) all along
the traversal algorithm. The effect of a merge or split is to add to $\sin$
a special edge $e$, precisely, the two edges representing
$e$ and the 2-sided enclosed face. Since
the two triangular faces incident to each side of $e$ are not in
$\sin$,
 a merge or a split does not add any edge to $D_0'$, so $D_0'$
remains acyclic. Now consider a conquest of a free boundary corner $b$.
Before the conquest, let $e$ and $e'$ be the edges delimiting $b$ in cw order,
let $f$ be the face encountered just before $e$ in cw order around the origin of $b$,
and let $v$ be the vertex of $D_0'$ corresponding to $f$.
Then, as shown in Figure~\ref{fig:dual_conq}, the effect of the conquest
on $D_0'$ is to attach a chain at $v$.
Hence $D_0'$
remains acyclic. At the end, $D_0'$ is equal to $D_0$, hence
$D_0$ is acyclic, so $G_0$ is cellular. Finally, $G_0$ has $n$
vertices ($G_0$ spans all vertices of $\cT$) and has $n+4g-1$
edges according to the local conditions. Since $G_0$ has genus
$g$, the Euler relation ensures that $G_0$ has $1+2g$ faces. The
proof for $G_1$ relies on the same arguments.
\end{proof}

\emph{Remark~5.} The properties of $G_2$ (cut-graph condition),
and of  $G_0$, $G_1$ (stated in Proposition~\ref{theo:maps}) can
be considered as extensions of the fundamental property of planar
Schnyder woods~\cite{Schnyder89,Schnyder90}: in the planar case,
for each color $i\in\{0,1,2\}$, the graph formed by the edges in
color $i$ plus the two outer edges incident to $v_i$ is a spanning tree.
Figure~\ref{fig:g_spanning_trees} shows an example in genus~$1$.

\begin{figure}[th]
\centering \scalebox{0.85}{
\includegraphics{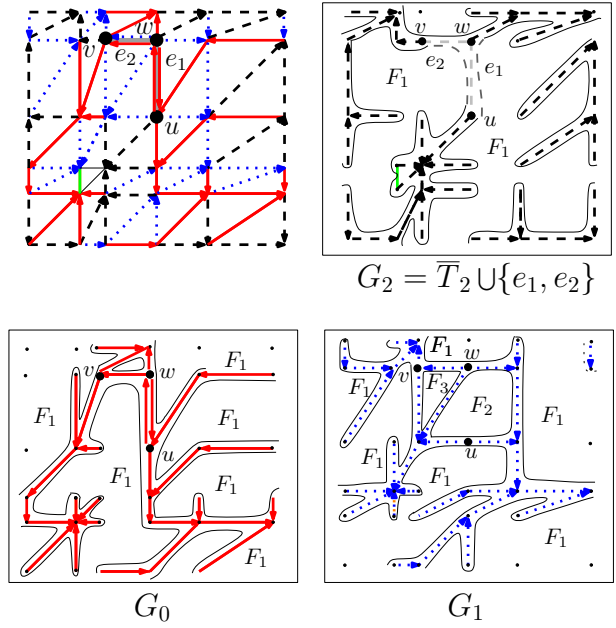}}
\caption{A triangulated torus endowed with a Schnyder
wood.
The dashed edges (color $2$) form a tree $T_2$, and the addition of the two
special edges and the two outer edges incident to $v_2$
yields a cut-graph $G_2$.
The solid edges (color $0$) plus the two outer edges incident to
$v_0$ form a spanning cellular subgraph $G_{0}$ with $3$ faces
(one face having degree $2$).
Similarly, the solid edges (color $1$) plus the two outer edges
incident to $v_1$ form a spanning cellular subgraph $G_{1}$ with
$3$ faces.} \label{fig:g_spanning_trees}
\end{figure}


\section{Application to encoding}
\label{sec:encoding_application}

In the planar case, Schnyder woods yield a simple encoding
procedure for triangulations, as described in~\cite{He99} and more
recently in~\cite{BB07a}. Precisely, a planar Schnyder wood with $n$
 vertices is encoded by two parenthesis words $W,W'$ of respective lengths
 $2n-2$ and $2n-6$. Let $\ol{T_2}$ be the tree $T_2$ plus the two outer edges
incident to $v_2$. Call $\theta$ the corner incident to $v_2$ in the outer face.
The first word $W$ is the parenthesis word (also called Dyck word)
that encodes the tree
$\ol{T_2}$, that is, $W$ is obtained from a cw walk (i.e., the walker has
the infinite face on its right) around $\ol{T_2}$ starting at $\theta$,
writing an opening parenthesis at the first
traversal of
an edge of $\ol{T_2}$ (away from the root) and a closing parenthesis
at the second traversal (toward the root).
The second word $W'$ is obtained from the same walk around $T_2$,
but $W'$ encodes the edges
that are not in $\ol{T_2}$, i.e., the edges of color $0$ and $1$.
Precisely, during the traversal, write an opening parenthesis in $W'$ each time an outgoing edge in color $0$ is crossed  and write a closing parenthesis
in $W'$ each time an ingoing edge of color $1$ is crossed.

For a triangulation with $n$ vertices, $W$ has length $2n-2$, and $W'$
has length $2n-6$. Hence the coding word
has total length $4n-8$. This code is both simple and quite compact, as
the length $4n-8$ is not far from the information-theory lower bound
of $\log_2 \left(4^4/3^3 \right) \approx
3.245$ bits per vertex, which is attained in the planar case by a
bijective construction due to Poulalhon and
Schaeffer~\cite{Pou03}.

In the higher genus case there does not exist an exact enumeration
formula, nevertheless an asymptotic estimate~\cite{Gao93} of the
number of genus $g$ rooted triangulations with $n$ vertices leads
to the information theory lower bound of $3.245n+\Omega(g\log n)$, i.e.,
the exponential growth rate is the same in every genus.
For the higher genus case we do not yet know any linear time encoding
algorithm matching asymptotically the information theory bound, and a bijective
construction based on a special spanning tree is still to be
found. Nevertheless we can here extend to higher genus the simple
encoding procedure of~\cite{He99,BB07a} based on Schnyder woods.

\paragraph{Encoding in higher genus} To encode the Schnyder wood we proceed in a similar way as in
the planar case except that we have to deal with the special
edges. Let $\cT$ be a genus $g$ triangulation with $n$ vertices
endowed with a Schnyder wood computed by our traversal algorithm;
precisely, we use the implementation described in
Lemma~\ref{lem:time_exec}. Let $\ol{T_2}$ be the spanning tree of
$\cT$ consisting of the edges in color $2$ plus the two edges
$\{v_0,v_2\}$ and $\{v_1,v_2\}$. Let $G_2$ be the cut-graph of the
Schnyder wood, i.e., $G_2$ is $\ol{T_2}$ plus the $2g$ special
edges. We classically encode $G_2$ as the Dyck word $W$ for
$\ol{T_2}$, augmented by $2g$ memory blocks, each of size
$O(\log(n))$ bits, so as to locate the two extremities of each
special edge. In each memory block we also store the colors and
directions of the two sides of the special edge. Hence $G_2$ is
encoded by a word $W$ of length $2n-2+O(g\log(n))$. The encoding
of the Schnyder wood is completed by a second binary word
 $W'$ that is obtained from a clockwise walk
along the (unique) face of $G_2$ (cw means that the face is on the right of the walker)  starting at the corner $\theta$ incident to $v_2$
in the root-face. Along this walk, we write a $0$ when crossing
a non-special outgoing
edge of color $0$ and we write a $1$ when crossing a non-special
ingoing edge of color $1$.
Since there are $2n-6+4g$ non-special edges of color $0$ or $1$, the
word $W'$ has length $2n-6+4g$. Therefore the pair of words $(W,W')$ is of
total length $4n+O(g\log(n))$. In addition these words can be obtained
in time $O((n+g)g)$ from a Schnyder wood on $\cT$ (as we have seen in
Lemma~\ref{lem:time_exec},
the Schnyder wood itself can be computed in time $O((n+g)g)$.

Now we are going to show that the pair $(W,W')$ actually encodes the Schnyder
wood (and in particular the triangulation) and that the Schnyder wood can
be reconstructed from $(W,W')$ in time $O((n+g)g)$. The proof relies
on two lemmas.

\begin{lemma}\label{lem:color0}
Let $\cT$ be a triangulation endowed with a $g$-Schnyder wood.
Then
the Schnyder wood can be recovered after the deletion process
that consists in removing all
the non-special edges of color 0. In other words, the information given by non-special edges of color 0 is
redundant.
\end{lemma}
\begin{proof}
To have a unified treatment (no special case for the vertex $v_0$)
it proves convenient here to direct the edges $\{v_0,v_2\}$ and
$\{v_0,v_1\}$ out of $v_0$ and to give color $2$ to $\{v_0,v_2\}$ and color
$1$ to $\{v_0,v_1\}$.
Consider a maximal non-empty interval $I$ of non-special edges of color $0$
going into a vertex $v$ of $\cT$. Let $e$ and $e'$ be the edges that
respectively precede and follow $I$ in cw order around $v$.
By the local conditions of Schnyder woods (Figure~\ref{fig:def_grealizer}(b)),
$e'$ is outgoing of color $1$; and either $e$ belongs to a special edge and is
ingoing of color $0$,
or $e$ is outgoing of color $2$.
Let $P=v_0,v_1,\ldots,v_k,v_{k+1}$ be the path of $\cT$ formed by the neighbors of $v$
in cw order between $e$ and $e'$, that is,
 $v_0$ is the other end of $e$, $v_{k+1}$ is the other end of $e'$,
and the $v_i$'s for $1\leq i \leq k$ are the other ends of the edges of $I$ taken in cw order around $v$. Then, by the local conditions of Schnyder woods,
each edge
 $\{v_i,v_{i+1}\}$, for $0\leq i \leq k$,  either is of color $1$ directed from $v_i$ to $v_{i+1}$ or is of color $2$ directed from $v_{i+1}$ to $v_i$. Hence, the edges of $P$ and the edges $e$ and $e'$
are not removed by the deletion process. Call $M$ the map created from $\cT$
by the deletion process.
Then there is a face $f$ in $M$ delimited by $P$, $e$ and $e'$: this is the face of $M$  formed by the removal of the edges in $I$.
  In addition the corner formed by $e$
and $e'$ is the unique corner of $f$ whose right-edge (looking toward the interior of $f$)
is outgoing of color $1$. Thus the edges removed inside $f$ (and more generally all the
removed edges)
 can be recovered:
one looks for the unique corner of $f$ whose right-edge is outgoing of color $1$,
and then one inserts an interval of ingoing edges of color $0$ at the corner so
as to triangulate $f$.
\end{proof}


\begin{figure}[th]
\centering \scalebox{0.90}{
\includegraphics{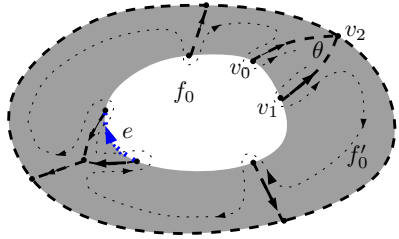}}
\caption{An edge $e$ colored $1$ (dotted arc) has the boundary
face $f_0$ on its right just before the conquest coloring $e$.
Hence, just before the conquest, a cw walk around $f_0'$ (dashed
lines) encounters the outgoing brin of $e$ first.}
\label{fig:edge1}
\end{figure}

\begin{lemma}\label{lem:red_encode}
Consider a $g$-Schnyder wood $S$ calculated
by the traversal algorithm under the implementation
described in Lemma~\ref{lem:time_exec}.
Denote by $G_2$ the cut-graph of $S$ and by $\theta$ the
 corner incident to $v_2$ in the root-face ($\theta$ is also a corner of $G_2$).
 Let $e$ be a non-special edge of color $1$ of $S$.

 Then, during a cw walk along $G_2$ (i.e., with the unique face of $G_2$ on the
 right of the walker) starting at $\theta$, the outgoing brin of $e$  is crossed
 before the ingoing brin of $e$.
\end{lemma}
\begin{proof}
At each step of \textsc{ComputeSchnyderAnyGenus} strictly before termination,
let $f_0$ be the boundary face of $\sin$ containing $\{v_0,v_1\}$ and let
$f_0'$ be the corresponding face of $\oC$ (we use the notation of
Lemma~\ref{lem:cutGraph},
$\oC$ consists of the edges and vertices of $G_2$ that are in $\sin$), that is,
$f_0'$ is the face of $\oC$ containing $f_0$.
An edge $e$ of color $1$ has $f_0$ on its right just before the
conquest coloring $e$ (by definition of the colorient rule).
Hence, as shown in Figure~\ref{fig:edge1}, $e$ is encountered first at its outgoing brin during a cw walk
around $f_0'$ starting at $\theta$; and this property will continue to hold for $e$
until the end of the traversal.

\end{proof}

We can now describe how to reconstruct the Schnyder wood from
the two words $(W,W')$. First,  construct the cut-graph $G_2$ using $W$.
Note that the directions of  edges and colors of the two sides of each special
edge of $G_2$ are known from $W$. Hence, by the local conditions of Schnyder woods, we can
already insert the outgoing brins of color $0$  or $1$ that are non-special
(a non-special brin is a brin of a non-special edge).
The non-special
outgoing brins of color $0$ are ordered as $b_1,b_2,\ldots,b_k$ according
to the order in which they are crossed during a cw walk along $G_2$
(i.e., with the unique face of $G_2$ on the right of the walker).
Next, the word $W'$ indicates where to insert
the non-special ingoing brins of color $1$. Precisely, factor $W'$ as
$$
W'=1^{\!r_1}01^{\!r_2}01^{\!r_3}\ldots 01^{\!r_{k+1}},
$$
where the integers $r_i$'s are allowed to be zero. Then, for each $i\in[1..k]$,
insert $r_i$ ingoing brins of color $1$ in the corner $(b_i,\mathrm{follower}(b_i))$
(where the follower of a brin $b$ is the next brin after $b$ in cw order around its origin).
And insert $r_{k+1}$ ingoing brins of color $1$
in the corner incident to $v_1$ delimited to the right by $\{v_1,v_0\}$.

Afterwards, we use
Lemma~\ref{lem:red_encode} to form the non-special edges of color $1$.
Write a parenthesis word $\pi$
obtained from a cw walk along $G_2$ starting at $\theta$,
writing an opening parenthesis each
time a non-special outgoing brin of color $1$ is crossed and writing a closing
parenthesis each
time a nonspecial ingoing brin of color $1$ is crossed.
Then, Lemma~\ref{lem:red_encode} ensures that the matchings of $\pi$
correspond to the non-special edges of color $1$ in the Schnyder wood, so
we just have to form the non-special
edges of color $1$ according to the matchings of $\pi$.

Finally, since the edges of color $0$ are redundant (by Lemma~\ref{lem:color0}),
there is no ambiguity to insert the edges of color $0$ at the end (i.e.,
complete the already inserted outgoing half-edges of color $0$ into
edges).

To conclude, the non-special edges of color $0$ are redundant, the
cut-graph can be encoded by a parenthesis word $W$ of length $2n-2$
(for the tree $\ol{T_2}$) plus $O(g\log(n))$ bits of
memory for the special edges, and the edges of color $1$ can be
inserted from a word $W'$  of length $2n-6+4g$.
Clearly the reconstruction of the Schnyder wood from $(W,W')$
takes time $O((n+g)g)$, since it just consists
in building the cut-graph $G_2$ and walking cw along $G_2$.
All in all, we obtain the following
result:

\begin{proposition}
A triangulation of genus $g$ with $n$ vertices can be encoded---via a $g$-Schnyder wood---by a binary word of length
$4n+O(g\log(n))$. Coding and decoding can be done in time $O((n+g)g)$.
\end{proposition}

We mention that one could also design a more sophisticated code
that supports queries, as done in~\cite{Chu98,BarbayISAAC07}. The
arguments would be similar to the ones given
in~\cite{BarbayISAAC07}, which treats plane (labeled)
triangulations.
To wit, given a genus $g$ (unlabeled) triangulation $\mathcal{T}$
with $f$ faces and $e$ edges, one could obtain a compact
representation of $\mathcal{T}$ using asymptotically $(2\log
6)e+O(g\log e)$ bits, or equivalently $7.755\ \!f+O(g\log f)$
bits, which answers queries for vertex adjacency and vertex degree
in $O(1)$ time.
The main idea would be to compute a g-Schnyder wood of
$\mathcal{T}$ and to encode the corresponding maps
$G_i$, $i\in\{0,1,2\}$.
In order to efficiently support adjacency queries on vertices, we
would have to encode the three maps $G_0$, $G_1$, $G_2$ using
 a multiple parenthesis system (3
types of parentheses).

In~\cite{CastelliWADS05} is described another partitioning strategy (not based
on Schnyder woods nor canonical orderings) answering queries, which
 achieves a better compression rate of $2.175f+O(g\log f)$ bits when dealing with genus $g$
 triangulations having $f$ triangles
(using a different face-based navigation).
Nevertheless, we believe that, compared to~\cite{CastelliWADS05},
an approach based on Schnyder woods would make it possible to deal
in higher genus with more general graphs~(\cite{Chu98}) and
\emph{labeled} graphs~(as done in~\cite{BarbayISAAC07} in the
planar case).


\section{Conclusion and perspectives}
\label{sec:conclusion}

We have extended to arbitrary genus the definition of Schnyder woods,
a traversal procedure for computing such a Schnyder wood in linear  time (for fixed
genus) and an encoding algorithm providing an asymptotic compression rate of $4$ bits
per vertex (again for fixed genus).
Some further problems and related topics are listed next.

\subsubsection*{Applications of Schnyder Woods as canonical orderings}

We point out that our graph traversal procedure induces an
ordering for treating the vertices so as to shell the surface
progressively. Such an ordering is already well known in the
planar case under the name of \emph{canonical ordering} and has
numerous applications for graph encoding and graph
drawing~\cite{Chu98,Kant96}. It is thus of interest to extend this
concept to higher genus.
%
%
The only difference is that in the genus $g$ case there is a small
number ---at most $2\cdot 2g$--- of vertices that might appear
several times in the ordering; these correspond to the vertices
incident to the $2g$ special edges (split/merge edges) obtained
during the traversal.
There are several open questions we think should be investigated
concerning the combinatorial properties of such orderings and the
corresponding edge orientations and colorations.
A related question in our context is to ask if
 any  Schnyder wood can be obtained as a result of our traversal procedure
 (if not, which property the Schnyder wood has to satisfy).
Another line of research is to see whether such an
ordering would yield an efficient algorithm for drawing a graph on a
genus $g$ surface (as it has been done in the planar
case~\cite{Kant96}).

\subsubsection*{Further extensions} Our approach relies on quite
general topological and combinatorial arguments, so the natural
next step should be to apply our methodology to other interesting
classes of graphs (not strictly triangulated), which have similar
characterization in the planar case.
Our topological traversal could be extended to the $3$-connected
case, precisely to embedded 3-connected graphs with face-width larger than 2, which correspond to polygonal meshes of genus $g$. We point out
that our encoding proposed in
Section~\ref{sec:encoding_application} could take advantage of the
existing compact encodings of planar
graphs~\cite{Chu98,Chi01,He99}, using similar parenthesis-based
approaches.

\subsubsection*{Lattice structure and graph encoding applications}
From the combinatorial point of view it should be of interest to
investigate whether edge orientations and  colorations in genus
$g$ have nice lattice properties, as in the planar case.
In the planar case, so-called minimal $\alpha$-orientations  have a deep
combinatorial role (they yield bijective constructions for several families
of planar maps, including triangulations), and as such,
  have also applications in graph
 drawing, random sampling,
  and coding~\cite{Pou03}.

In the planar case, as shown by Brehm~\cite{Brehm_thesis},
the minimal Schnyder wood is reached by a
``left-most driven'' traversal of the triangulation,
and is computable in linear time.
We would like to extend these principles to any genus
and derive from it a linear time encoding procedure with (asymptotically)
optimal compression rate. Hopefully these principles can also be applied
to polygonal  meshes of arbitrary genus.


\subsubsection*{Acknowledgments.}
We are grateful to Nicolas Bonichon, Cyril Gavoille and Arnaud Labourel for very
interesting discussions on Schnyder woods.
First author would like to thank \'Eric Colin de Verdi\`ere for
pointing out some useful topological properties of graphs on
surfaces.
We are extremely grateful to Olivier Bernardi, Guillaume Chapuy, and
Gilles Schaeffer for
enlightening discussions on the combinatorics of maps that
motivated this work and helped to clarify the ideas.
We finally thank the reviewers for their very insightful remarks.
The first two authors' work was partially supported by ERC
research starting grant "ExploreMaps".
The last author would like to thank CNPq and FAPERJ for financial support.


\bibliographystyle{abbrv}


\begin{thebibliography}{10}

\bibitem{BarbayISAAC07}
J.~Barbay, L.~Castelli-Aleardi, M.~He, and J.~I. Munro.
\newblock Succinct representation of labeled graphs.
\newblock In {\em ISAAC}, pages 316--328, 2007.

\bibitem{BB07a}
O.~Bernardi and N.~Bonichon.
\newblock  Intervals in Catalan lattices and realizers of triangulations.
\newblock {\em Journal of Combinatorial Theory - Series A}, Vol 116(1) pp 55-75, 2009.

\bibitem{Bonichon2002_thesis}
N.~Bonichon.
\newblock {\em Aspects algorithmiques et combinatoires des r\'ealiseurs des
  graphes plans maximaux}.
\newblock PhD thesis, Bordeaux I, 2002.

\bibitem{Bon03}
N.~Bonichon, C.~Gavoille, and N.~Hanusse.
\newblock An information-theoretic upper bound of planar graphs using
  triangulation.
\newblock In {\em STACS}, pages 499--510. Springer, 2003.

\bibitem{Bon06}
N.~Bonichon, C.~Gavoille, N.~Hanusse, D.~Poulalhon, and
G.~Schaeffer.
\newblock Planar graphs, via well-orderly maps and trees.
\newblock {\em Graphs and Combinatorics}, 22(2):185--202, 2006.

\bibitem{Bon05}
N.~Bonichon, C.~Gavoille, and A.~Labourel.
\newblock Edge partition of toroidal graphs into forests in linear time.
\newblock In {\em ICGT}, volume~22, pages 421--425, 2005.

\bibitem{Brehm_thesis}
E.~Brehm.
\newblock $3$-orientations and Schnyder-three tree decompositions.
\newblock Master's thesis, Freie Universit\"at Berlin, 2000.

\bibitem{CabelloM07}
S.~Cabello and B.~Mohar.
\newblock Finding shortest non-separating and non-contractible cycles for
  topologically embedded graphs.
\newblock {\em Discrete {\&} Comp. Geometry}, 37(2):213--235, 2007.

\bibitem{CastelliWADS05}
L.~Castelli-Aleardi, O.~Devillers, and G.~Schaeffer.
\newblock Succinct representation of triangulations with a boundary.
\newblock In {\em WADS}, pages 134--145. Springer, 2005.

\bibitem{CastelliDS06}
L.~Castelli-Aleardi, O.~Devillers, and G.~Schaeffer.
\newblock Succinct representations of planar maps.
\newblock In {\em Theoretical Computer Science}, 408:174-187, 2008.
\newblock (preliminary version in {\em SoCG'06})

\bibitem{Ch}
G. Chapuy.
\newblock Asymptotic enumeration of constellations and related families of maps on orientable surfaces.
\newblock {arXiv:0805.0352}, 2008.

\bibitem{ChMaSc}
G. Chapuy, M. Marcus, and G. Schaeffer.
\newblock A bijection for rooted maps on orientable surfaces.
\newblock {arXiv:0712.3649}, 2007.

\bibitem{Chi01}
Y.-T. Chiang, C.-C. Lin, and H.-I. Lu.
\newblock Orderly spanning trees with applications to graph encoding and graph
  drawing.
\newblock In {\em SODA}, pages 506--515, 2001.

\bibitem{Chu98}
R.~C.-N. Chuang, A.~Garg, X.~He, M.-Y. Kao, and H.-I. Lu.
\newblock Compact encodings of planar graphs via canonical orderings and
  multiple parentheses.
\newblock {\em ICALP}, pages 118--129, 1998.

\bibitem{DeFraysseix01}
H.~de~Fraysseix and P.~O. de~Mendez.
\newblock On topological aspects of orientations.
\newblock {\em Discrete Mathematics}, 229:57--72, 2001.

\bibitem{Ossona1994_thesis}
P.~O. de~Mendez.
\newblock {\em Orientations bipolaires}.
\newblock PhD thesis, Paris, 1994.

\bibitem{VerdiereL02}
{\'E}.~C. de~Verdi{\`e}re and F.~Lazarus.
\newblock Optimal system of loops on an orientable surface.
\newblock In {\em Discrete {\&} Computational Geometry}, 33(3):507-534, 2005.
\newblock (preliminary version in {\em FOCS'02})

\bibitem{EricksonH04}
J.~Erickson and S.~Har-Peled.
\newblock Optimally cutting a surface into a disk.
\newblock {\em Discrete {\&} Computational Geometry}, 31(1):37--59, 2004.

\bibitem{Felsner01}
S.~Felsner.
\newblock Convex drawings of planar graphs and the order dimension of
  $3$-polytopes.
\newblock {\em Order}, 18:19--37, 2001.

\bibitem{Felsner04}
S.~Felsner.
\newblock Lattice structures from planar graphs.
\newblock {\em Electronic Journal of Combinatorics}, 11(15):24, 2004.

\bibitem{Fusy2007_thesis}
\'E.~Fusy.
\newblock {\em Combinatoire des cartes planaires et applications
  algorithmiques}.
\newblock PhD thesis, Ecole Polytechnique, 2007.

\bibitem{Fus05}
\'E.~Fusy, D.~Poulalhon, and G.~Schaeffer.
\newblock Dissections, orientations, and trees with applications to optimal mesh encoding and random
sampling.
\newblock In {\em ACM Transactions on Algorithms}, 4(2), 2008.
\newblock (preliminary version in {\em SODA'05})

\bibitem{Gao93}
Z. Gao.
\newblock A pattern for the asymptotic number of rooted maps on surfaces.
\newblock {\em Journal of Combinatorial Theory, Series A}, 64:246--264, 1993.

\bibitem{He99}
X.~He, M.-Y. Kao, and H.-I. Lu.
\newblock Linear-time succinct encodings of planar graphs via canonical
  orderings.
\newblock {\em SIAM Journal on Discrete Mathematics}, 12:317--325, 1999.

\bibitem{Kant96}
G.~Kant.
\newblock Drawing planar graphs using the canonical ordering.
\newblock {\em Algorithmica}, 16(1):4--32, 1996.

\bibitem{Kee95}
K.~Keeler and J.~Westbrook.
\newblock Short encodings of planar graph and maps.
\newblock {\em Discrete and Applied Mathematics}, 58:239--252, 1995.

\bibitem{Kutz06}
M.~Kutz.
\newblock Computing shortest non-trivial cycles on orientable surfaces of
  bounded genus in almost linear time.
\newblock In {\em SoCG}, pages 430--438, 2006.

\bibitem{LazarusPVV01}
F.~Lazarus, M.~Pocchiola, G.~Vegter, and A.~Verroust.
\newblock Computing a canonical polygonal schema of an orientable triangulated
  surface.
\newblock In {\em SoCG}, pages 80--89, 2001.

\bibitem{LewinerLRV04}
T.~Lewiner, H.~Lopes, J.~Rossignac, and A.~W. Vieira.
\newblock Efficient edgebreaker for surfaces of arbitrary topology.
\newblock In {\em Sibgrapi}, pages 218--225, 2004.

\bibitem{LewinerLT03}
T.~Lewiner, H.~Lopes, and G. Tavares.
\newblock Optimal discrete Morse functions for 2-manifolds  .
\newblock {\em Computational Geometry}, 26(3): 221-233, 2003.

\bibitem{LopesRSST03}
H.~Lopes, J.~Rossignac, A.~Safonova, A.~Szymczak, and G.~Tavares.
\newblock Edgebreaker: a simple implementation for surfaces with handles.
\newblock {\em Computers {\&} Graphics}, 27(4):553--567, 2003.

\bibitem{Mohar_book}
B.~Mohar and C.~Thomassen.
\newblock {\em Graphs on Surfaces}, Johns Hopkins University Press, 2001.

\bibitem{Pou03}
D.~Poulalhon and G.~Schaeffer.
\newblock Optimal coding and sampling of triangulations.
\newblock {\em Algorithmica}, 46:505--527, 2006.

\bibitem{Ros99}
J.~Rossignac.
\newblock Edgebreaker: Connectivity compression for triangle meshes.
\newblock {\em Transactions on Visualization and Computer Graphics}, 5:47--61,
  1999.

  \bibitem{Rou72}
C. Rourke and B. Sanderson.
\newblock Introduction to piecewise-linear topology,
  1972.

  \bibitem{Schaeffer_thesis}
G.~Schaeffer.
\newblock {\em Conjugaison d'arbres et cartes combinatoires al\'eatoires}.
\newblock PhD thesis, Bordeaux I, 1999.

\bibitem{Schnyder89}
W.~Schnyder.
\newblock Planar graphs and poset dimension.
\newblock {\em Order}, pages 323--343, 1989.

\bibitem{Schnyder90}
W.~Schnyder.
\newblock Embedding planar graphs on the grid.
\newblock In {\em SODA}, pages 138--148, 1990.

\bibitem{Tarjan72}
R. E. Tarjan.
\newblock Depth first search and linear graphs algorithms.
\newblock {\em SIAM Journal of Computing}, 1:146--160, 1972.

\bibitem{Tarjan74}
R. E. Tarjan.
\newblock A note on finding the bridges of a graph.
\newblock {\em Information Processing Letters}, 2:160--161, 1974.

\bibitem{Tur84}
G.~Turan.
\newblock On the succinct representation of graphs.
\newblock {\em Discrete {\&} Applied Mathematics}, 8:289--294, 1984.

\bibitem{Tutte63}
W.~Tutte.
\newblock A census of planar maps.
\newblock {\em Canadian Journal of Mathematics}, 15:249--271, 1963.

\bibitem{VegterY90}
G.~Vegter and C.-K. Yap.
\newblock Computational complexity of combinatorial surfaces.
\newblock In {\em SoCG}, pages 102--111, 1990.

\end{thebibliography}

%
%


\end{document}